\documentclass[a4paper]{article}

\usepackage[english]{babel}					
\usepackage{enumerate,enumitem}	 
\usepackage{amsmath,amsthm,amssymb,mathrsfs,amsfonts,calc}
\usepackage{algorithm,algorithmic}
\usepackage{cleveref}
\usepackage{url}
\usepackage{graphicx,epstopdf}
\usepackage{pgfplotstable}
\usepackage{subcaption}
\usepackage{subfloat}

\DeclareMathOperator*{\psnr}{PSNR}
\DeclareMathOperator*{\snr}{SNR}
\DeclareMathOperator*{\argmin}{argmin}

\DeclareMathOperator*{\divergence}{div}
\DeclareMathOperator*{\sinc}{sinc}

\makeatletter
\DeclareOldFontCommand{\rm}{\normalfont\rmfamily}{\mathrm}
\DeclareOldFontCommand{\sf}{\normalfont\sffamily}{\mathsf}
\DeclareOldFontCommand{\tt}{\normalfont\ttfamily}{\mathtt}
\DeclareOldFontCommand{\bf}{\normalfont\bfseries}{\mathbf}
\DeclareOldFontCommand{\it}{\normalfont\itshape}{\mathit}
\DeclareOldFontCommand{\sl}{\normalfont\slshape}{\@nomath\sl}
\DeclareOldFontCommand{\sc}{\normalfont\scshape}{\@nomath\sc}
\makeatother

\newcommand{\norm}[1][\cdot]{\left\| #1\right\|}
\newcommand{\scalarproduct}[2]{\left\langle #1,#2\right\rangle}

\newcommand{\A}{\mathscr{A}}
\newcommand{\T}{\mathscr{T}}
\newcommand{\R}{\mathbb{R}}
\renewcommand{\P}{\mathscr{P}}

\newcommand{\F}{\mathscr{F}}
\newcommand{\Reg}{\mathscr{R}} 
\newcommand{\X}{\mathcal{X}}
\newcommand{\Y}{\mathcal{Y}}
\newcommand{\Sob}{H}

\newcommand{\rmd}{\textnormal{d}}

\newtheorem{theorem}{Theorem}[section]
\newtheorem{lemma}[theorem]{Lemma}
\newtheorem{proposition}[theorem]{Proposition}

\newtheorem{corollary}[theorem]{Corollary}

\title{A class of regularizations based on nonlinear isotropic diffusion for inverse problems}
\author{$^\dag$Bernadette N. Hahn, $^\ddag$Ga\"{e}l Rigaud, $^\dag$Richard Schm\"{a}hl}
\date{\small
    $^\dag$Department of Mathematics, University of Stuttgart, D-70569 Stuttgart, Germany\\%
    $^\ddag$Center for Industrial Mathematics, University of Bremen, D-28344 Bremen, Germany\\[2ex]%
}
\begin{document}

\maketitle 
    
\begin{abstract}
Building on the well-known total-variation (TV), this paper develops a general regularization technique based on nonlinear isotropic diffusion (NID) for inverse problems with piecewise smooth solutions. The novelty of our approach is to be adaptive (we speak of A-NID) \textit{i.e.} the regularization varies during the iterates in order to incorporate prior information on the edges, deal with the evolution of the reconstruction and circumvent the limitations due to the non-convexity of the proposed functionals. After a detailed analysis of the convergence and well-posedness of the method, this latter is validated by simulations perfomed on computerized tomography (CT).
\end{abstract}


\section{Introduction}

Many image processing tasks as well as industrial and medical applications ask for the solution of an inverse problem, \textit{i.e.}
\begin{center} \it
find $f$ from $g^\varepsilon$ with $\Vert g^\varepsilon - \A f\Vert_{\Y} \leq \epsilon$
\end{center}
in which $\A : \X \to \Y$ stands for an observation mapping between two function spaces $\X$ and $\Y$ and characterizes the nature of the application. In photography and image restoration, $\A$ is often assimilated to a convolution operator characterizing, for instance, the blur produced by an out-of-focus camera. In non-destructive testing and imaging, typically in Computerized Tomography (CT), $\A$ describes the evolution of a wave by the medium under study.  The inverse problem consists then in reconstructing a characteristic function of the medium depending on the nature of the waves.

An important aspect of inverse problems, in particular when $\A$ is compact, is their ill-posed nature, \textit{i.e.} their solution given by the Moore-Penrose inverse, does not depend continuously on the data. Regularization is thus essential to deal with this unstability issue and aims to build a family of continuous operators which approximate the Moore-Penrose inverse. While many different approaches were developed to build such a regularization, we focus in this manuscript on the solutions of   unconstrained minimization problems under the form
\begin{equation}
f^\gamma = \argmin_{f\in\X} \left\lbrace\frac12 \norm[\A f -g]^2_\Y + \Reg_\gamma(f) \right\rbrace
\end{equation}
where the penalty term $\Reg_\gamma$ constraints the solution space according to some \textit{a priori} on the solution: smoothness, sparsity, piecewise constancy, etc. Regularization functionals of this form are generally called \textit{Tikhonov functionals}.
An important feature of images is the contrast and the sharpness of the edges within. At this aim, the total-variation functional $TV(f)$ was introduced in \cite{rudin92} and has become one of the standard penalty terms in image processing and imaging. Typically, the least-square minimization problem constrained by TV is solved by a primal-dual approach in the discrete setting, see for instance \cite{chambolle11}. However, assuming the solution space to be a subset of  $\Sob^1(\Omega)$, $\Omega \subset \mathbb{R}^n$ a compact set, the solution $f_{TV}^\gamma$ obtained by a total-variation regularization can be seen as a minimizer of the special case $p(s) = \sqrt{s}$ of 
\begin{equation}
	\label{eq:general_representation_PM}
\min_{f\in\X} \frac12 \norm[\A f -g]^2_\Y + \gamma \int_\Omega p\left(\norm[\nabla f(x)]_2^2\right) \mathrm{d}x
\end{equation}
where $\norm[\cdot]_2$ denotes the euclidean norm. A gradient descent is then possible by approximating $p_{TV}$ by the continuously differentiable $p_{AV}(s) = \sqrt{\epsilon+s}$, $\epsilon>0$, as proposed in \cite{acar94}.
This notation has the advantage to be more flexible for the construction of a suited function $p$ at the cost of a regularity condition on the solution space. 

\subsection{Regularization meets diffusion equations}

\par The minimization problem (\ref{eq:general_representation_PM}) was studied for different functions $p$, see \cite{scherzer00} or \cite{teboul98}. One of the motivation of the representation (\ref{eq:general_representation_PM}) is the interpretation of the regularization term with diffusion equations. In physics, diffusion processes rule the evolution of concentration $u$ and satisfy
$$
\partial_t u=\divergence \left( D \cdot \nabla u \right)
$$
where $D$ is a diffusion tensor characterizing the nature of diffusion. An interesting case is $D \equiv 1$ which leads to the heat equation
\begin{equation}\label{eq:classical-heat-equation}
\left\lbrace
\begin{array}{rcl}
\partial_t u(t,x) &=& \Delta u \\
u(0,x) & = & u_0.
\end{array}
\right.
\end{equation}
Solutions to \cref{eq:classical-heat-equation} can be interpreted by a smoothing operator with increasing smoothness over time. From a local point of view, one can speak of \textit{forward diffusion}. The inverse heat equation consists in reverting the process by solving $\partial_t u(t,x) = - \Delta u $. In this case, one can speak of \textit{backward diffusion}. However, this inverse problem is exponentially ill-posed, see for instance \cite{rieder03}.  

In order to exploit local backward diffusions in order to sharpen an image, Perona and Malik proposed in \cite{perona90} the following nonlinear isotropic model   
\begin{equation}\label{eq:classical-perona-malik}
\left\lbrace
\begin{array}{rcl}
\partial_t u(t,x) &=& \divergence \left(\varphi(\norm[\nabla u]_2^2)\nabla u\right) \\
u(0,x) & = & u_0
\end{array}
\right.
\end{equation}
with $\varphi :\mathbb{R}^+\to \mathbb{R}^+$ a decreasing function. Evaluating the divergence operator in (\cref{eq:classical-perona-malik}) leads to 
\begin{equation*}
\partial_{t}u = \varphi\left(\norm[\nabla u]^2_2\right)u_{\xi\xi}+\psi\left(\norm[\nabla u]_2\right)u_{\eta\eta} \quad \mbox{with } 
\psi(s) = \frac{\mathrm{d}}{\mathrm{d} s} (s\cdot \varphi(s^2)),
\end{equation*}
$\xi \perp \nabla u$ and $\eta \; \| \; \nabla u$, see \textit{e.g.} \cite{weickert98,bredies11,prasath15}. Given an edge, the first term carried by $u_{\xi\xi}$ rules the diffusion in the tangent direction to the edge while the second term carried by $u_{\eta\eta}$ controls the diffusion in the normal direction. Since $\varphi$ is a positive function, the preservation of edges requires thus $\psi(s)\leq 0$ when $s\to \infty$. In order to be consistent between the principle of diffusion in physics and our regularization strategy, we will consider the following notations:
\begin{itemize}
\item $p:\R_0^+\to\R_0^+$ will denote the \textit{penalty function};
\item $\varphi(s):= p'(s)$ will stand for the \textit{diffusion function} and
\item $\psi(s) := s \cdot \varphi(s^2)$ will define the \textit{flux function}. 
\end{itemize}

	\par If the function is in some sense smooth enough s.t. the gradient exists (and stays bounded), then the sign of $\psi$ controls the behaviour of the diffusion equation at this edge depending on a certain threshold $\lambda>0$.
Perona and Malik proposed 	
$$
\varphi_1(s^2) = \frac{1}{1+\frac{s^2}{\lambda^{2}}}\quad \textnormal{and}\quad \varphi_2(s^2) =e^{-\frac{s^2}{\lambda^{2}}}.
$$	
The behaviour of the \textit{flux function} $\psi_2 (s) := s \cdot \varphi_2(s^2)$ is depicted in \Cref{fig:flux_PM}. Here the parameter $\lambda$ controls the edge or contrast enhancement since all gradients $\nabla u < \lambda$ are subject to a forward diffusion (and thus a smoothing) while all gradients $\nabla u > \lambda$ are subject to backward diffusion (and thus edge enhancement). 

In the case of total-variation, a similar behaviour can be observed. First, due to the non-differentiability in 0 of the penalty function $p_{TV}(s) := \sqrt{s}$, it is possible to approximate this functional using the model proposed by Acar and Vogel \cite{acar94}, \textit{i.e.} $p_{AV}(s) := \sqrt{\epsilon+s}$. The corresponding flux function is then given by $\psi_{AV} (s) = s \cdot (\epsilon+s^2)^{-1/2}$ and is depicted in Figure \Cref{fig:flux_PM}. We observe that the small values of $\norm[\nabla u]$ are subject to a forward diffusion while the largest gradients reach a plateau. This plateau indicates that total-variation does not perform backward diffusion but instead avoids the forward diffusion in the normal direction to the contours. 

\begin{figure}\centering
\includegraphics[width=0.8\linewidth]{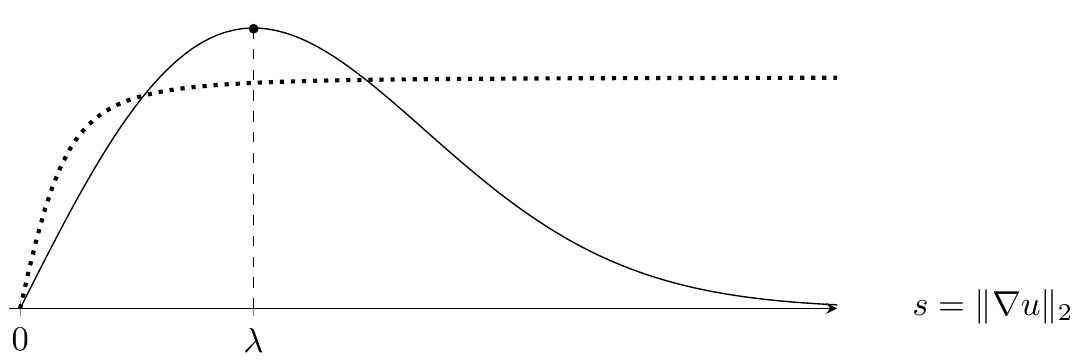}
\caption{General behaviour of the flux function associated to $\varphi_2$ (solid) and of $\varphi_{AV}$ (dotted).\label{fig:flux_PM}	}
\end{figure}

Total-Variation has the advantage to provide satisfactory image restoration without \textit{a priori} information about the target image. On the other hand, the Perona-Malik model \cite{perona90} has shown impressive results on edge preservation with a suited contrast parameter $\lambda$ and has been extended to inverse problems with applications in image restoration, see \cite{scherzer00}. A disadvantage of the Perona-Malik model is to be sensitive to the sole parameter $\lambda$ when the complexity of an image makes intervene many different levels of contrast.

\subsection{Our proposed model}

In this paper, we propose to extend the Perona-Malik model to a larger class of nonlinear isotropic diffusion (NID) equations in order to handle the complexity of an image. For example in imaging, it is standard to use the prior knowledge on the intensities of the materials to help the post-processing and the image reconstruction from the measurement.
Therefore, one can assume to know \textit{a priori} the magnitude of the edges, noted $(s_k)_k$. For such images a suited penalty term will then produce backward diffusion around the edge magnitudes $s_k$ and forward diffusion otherwise leading to alternate forward and backward diffusion adequately with the prior structure of the target image. 
	\par Similar approaches were proposed in the literature: \cite{douiri07} and \cite{correia11} adapted the Perona-Malik equation for diffuse optical tomography,  \cite{roy18} to acousto-electric tomography. \cite{nordstroem90} considedred a "biased" diffusion process for edge detection, \cite{welk05} applied it to the deconvolution problem, and \cite{huang13} used a fourth-order version for super-resolution image reconstruction. While the latter two were based on partial differential equations, the former were very problem specific. Another combination of these methods can be found in \cite{scherzer00}, where the authors analyzed regularizations and diffusion processes for image denoising, that is the special case $\A$ being the identity operator.  In \cite{charbonnier97}, the authors gave conditions on the penalty term to perserve edges and considering a single optimization problem. However, due to the non-convexity of the penalty term, existence and stability of a solution were not achieved, see \cite{teboul98}. The relation between PDEs and gradient flow equations was introduced in \cite{prasath15}. 
	
\par In a first step, we propose to solve the inverse problem $\A f = g$, $\A:\X\to\Y$ linear and bounded, by the minimization problem
$$
\label{´}
\min_{f\in\X} \frac12 \norm[\A f -g]_{\Y}^2  + \sum_{k=1}^K \gamma_k \int_\Omega p_{k}(\norm[\nabla f(x)]_2^2) \mathrm{d}x + \frac{\alpha}{2} \norm[f]_{\X}^2 \quad \alpha,\gamma_k>0 .
$$
The use of $K$ different thresholds enables a better control on the condition and intensity of a forward/backward diffusion and thus provides a better trade-off between the smoothing process and the contrast-enhancement. As mentioned in \cite{li94}, the intensity of the diffusion -- determined by the range where $\varphi_{k}'$ is sufficiently smaller than zero --  is an important factor of the \textit{effectiveness} of the edge preserving regularization. This latter is an important advantage for the proposed approach over the original Perona-Malik functional.

\par
However, this model presents a well-known flaw. Even if $f\in \Sob^1(\Omega)$, the functional is not weakly lower semicontinuous due the missing convexity of $p_k$, making it difficult to analyse in terms of convergence as mentioned by \cite{scherzer00}. Based on \cite{catte92}, a solution was proposed in \cite{scherzer00} for the iterated Perona-Malik regularization scheme in image denoising. It consists in approximating the gradient using a suitable smoothing operator, typically a convolution operator with a normal distribution. In this manuscript, we follow this idea and will consider the operator 
\begin{equation*}
\nabla_\sigma f := \nabla G_\sigma \ast f \qquad \mbox{with} \quad 
G_\sigma (x) = \frac{1}{2\pi \sigma^2} \exp\left(- \frac{\norm[x]_2^2}{2\sigma^2} \right) 
\end{equation*}
instead of $\nabla f$. A second advantage of this approximation is to rule out the regularity on $f$ regarding $\Sob^1(\Omega)$ and allow our study for $f\in \X =  L_2(\Omega)$, the space of square-integrable functions. Our general NID minimization problem is thus
\begin{align}
	\label{eq:def_NID}
	\min_{f\in L_2(\Omega)}\ &\T_{g,\gamma,\sigma,\alpha}(f) := \frac12 \norm[\A f -g]_{\Y}^2  + \Reg_{NID}(f)  + \frac{\alpha}{2} \norm[f]_{ L_2(\Omega)}^2 \\
	&\mbox{with}\qquad \Reg_{NID}(f) = \sum_{k=1}^K \gamma_k \int_\Omega p_{k}(\norm[\nabla_{\sigma_k} f(x)]_2^2) \mathrm{d}x \nonumber
\end{align}
in which $\alpha,(\gamma_k)_k,(\sigma_k)_k$ are positive parameters and $\varphi_k$ -- the diffusion functions associated to $p_k$ -- are smooth positive functions. The construction of the function $\varphi_k$ will thus depend on the application and on the type of images in order to produce suited alternating forward-backward diffusions.  Not crucial in practice, the term $\norm[f]_{ L_2(\Omega)}^2 $ enforces coercivity of the functional. The well-posedness of the NID regularization method is studied in	\Cref{subsec:well-posedness-of-pm-reg} while a gradient descent algorithm is developed in \Cref{subsec:gradient-descent-for-pm-reg}.

\par A drawback of non-convex penalty terms as $\Reg_{NID}$ is to result often in local and not global minima. The initial value becomes then crucial to ensure local convergence to a global minimizer. Global convergence can be shown in some cases but fails in general. While these difficulties are inherent to our class of regularizations, we propose to relax the penalty term by allowing it to vary during the iteration process. This adaptive NID approach (A-NID) consists in applying a gradient descent on successive functionals of the form
\begin{equation}\label{eq:def_aNID}
\T^n_{g,\gamma,\sigma,\alpha}(f) := \frac12 \norm[\A f -g]_{\Y}^2  + \Reg_{NID}^n(f)  + \frac{\alpha}{2} \norm[f]_{ L_2(\Omega)}^2 
\end{equation}
\textit{i.e.}
\begin{equation*}
f^{n+1} = f^n - t_n \nabla_f \T^n_{g,\gamma,\sigma,\alpha}(f^n)
\end{equation*}
with $t_n$ a given step size and $\Reg_{NID}^n$ a varying NID-based penalty term. The convergence of this approach is shown in \Cref{sec:a-nid}.

\par Finally we provide in \Cref{sec:numerical-implementation} numerical tests of both static and adaptive NID regularizers for CT-data and compare them to the standard methods.

\section{Well-posedness of the NID regularizer} 		
\label{subsec:well-posedness-of-pm-reg}
	
In this section, we first establish existence of solutions to the optimization problem \cref{eq:def_NID}, as well as stability and convergence to a generalized version of minimal norm solutions. We denote by $C_{+,b}^m(\mathbb{R}^+_0)\subset C^m(\mathbb{R})$ the space of $m$-times continuously differentiable functions $f$ mapping $\mathbb{R}_0^+$ onto $\mathbb{R}_0^+$ s. t. $f,f',...,f^{(m-1)}$ are bounded. We will always asume that the $p_k$ are at least $C_{+,b}^1(\mathbb{R}_0^+)$ for $k=1,...,K$.
The following Lemma delivers the fundamental properties of the studied functional.  

	\begin{lemma}\label{lemma:pm-weakly-lower-semicont}
	The associated functional $\T_{g,\gamma,\sigma,\alpha}:L_2(\Omega)\to \mathbb{R}$ defined in \cref{eq:def_NID} is proper, bounded from below, coercive and  weakly (sequentially) lower semicontinuous.
	\end{lemma}

	\begin{proof}
	    That it is proper and bounded from below, follows immediately from the regularity and positivity constraints on $p_k$.
		The additive term $\frac{\alpha}{2} \norm[f]_{ L_2(\Omega)}^2 $ ensures that the functional $\T_{g,\gamma,\sigma,\alpha}$ is coercive. 
		Regarding the weakly lower semi-continuous property, we make use of the proof of [\cite{scherzer00}, Theorem 4]. There the authors showed lower semicontinuity for the regularized (standard) Perona-Malik functional 
		\begin{align*}
			\T_{\text{R-PM}}(f):=\norm[f-u]_{L_2(\Omega)}^2+h\int_{\Omega}\ln\left(1+\norm[\nabla_\sigma f]_2^2\right)\rmd x
		\end{align*}
		with $u\in L_2(\Omega),h\in \mathbb{R}^+$, which is a very specific case of our more general framework. Since $\A$ is linear and continuous (and hence maps weakly convergent sequences to weakly convergent sequences), the lower semicontinuity follows for the norm terms immediately.
		We now have to verify the weakly lower semicontinuity for our suggested regularization term $\Reg_{NID}$.\\
		Let $(f_n)_n \subset L_2(\Omega)$ be a sequence with weak limit $f^*$. Then, due to the compactness of $\nabla_{\sigma_k}:L^2(\Omega)\to C^1(\Omega)$, the sequence $\left(\nabla_{\sigma_k} f_n\right)_n$ converges uniformly to some $\nabla_{\sigma_k} f^*$ with $k\in\{1,\ldots,K\}$ arbitrary. 
		In particular $\left(\nabla_{\sigma_k} f_n\right)_n$ is bounded by some constant $C>0$ with respect to the sup norm. Since $p_k$ is in $C^1(\mathbb{R})$, it is Lipschitz continuous on $[0,C^2]$ and we find that $p_{k}\left(\norm[ \nabla_{\sigma_k} f_n]_2^2\right)\to p_{k}\left(\norm[ \nabla_{\sigma_k} f]_2^2\right)$ uniformly for all $k$ and thus
		\begin{equation}
			\label{eq:strong-conv}
			\int_{\Omega}p_k\left(\left\|\nabla_{\sigma_k} f_n\right\|^2_2\right)\to \int_{\Omega}p_k\left(\left\|\nabla_{\sigma_k} f\right\|^2_2\right),
		\end{equation} 
		The assertion then follows using a linearity argument.
	\end{proof}

Standard results in the calculus of variations, see \textit{e.g.} [\cite{bredies11}, Theorem 6.17], combined with \Cref{lemma:pm-weakly-lower-semicont} now imply the existence of solutions.
	\begin{theorem} Let $(p_k)_{k=1,...,K}\in C^1_+\left(\mathbb{R}_0^+\right)$. Then
		the functional $\T_{g,\gamma,\sigma,\alpha}:L_2(\Omega)\to \R$ 
		admits a minimizer for every choice of parameters $(\gamma_k)_k>0$, $(\sigma_k)_k>0$, $\alpha>0$ and for every $g\in \Y$.
	\end{theorem}

An important aspect of a minimization functional is to guarantee stability when the data $g$ suffer  some perturbations. Following on from \cite{acar94} which studied the case of standard TV, we first demonstrate uniform coercivity and consistency of the NID functional.

	\begin{proposition}\label{prop:consistency}
		Let $(g_\delta^n)_n\subseteq \Y$ be a sequence with (strong) limit $g\in \Y$.
		The corresponding sequence $\T_{g_\delta^n,\gamma,\sigma,\alpha}$ (with $\gamma,\alpha,\sigma$ fixed) is
		\begin{itemize}
			\item \textbf{uniformly coercive}, i.e. for all $(f_n)_n\subseteq L_2(\Omega)\textnormal{ with }\norm[f_n]_{L_2(\Omega)}\to \infty$ it holds
			\begin{equation}
			\label{eq:uniform-coercivity}
			\lim\limits_{n\to\infty} \T_{g_\delta^n,\gamma,\sigma,\alpha}(f_n)\to\infty, 
			\end{equation}
			\item \textbf{consistent}, i.e. for every $M>0$ and $\varepsilon>0$ there exists $N_\varepsilon\in \mathbb{N}$ with
			\begin{equation}
			\label{eq:functional-consistency}
			\left|\T_{g_\delta^n,\gamma,\sigma,\alpha}(f)-\T_{g,\gamma,\sigma,\alpha}(f)\right|\leq \varepsilon\quad\forall n\geq N_\varepsilon
			\end{equation}
			for all  $f\in L_2(\Omega)$ with $\norm[f]_{L_2(\Omega)}\leq M$.
		\end{itemize}
	\end{proposition}

	\begin{proof} 
	The uniform coercivity property is again induced by the penalty term $\frac{\alpha}{2}\norm[f]_{L_2(\Omega)}^2$.
		Regarding the consistency, it holds
		\begin{align*}				\left|\T_{g_\delta^n,\gamma,\sigma,\alpha}(f)-\T_{g,\gamma,\sigma,\alpha}(f)\right|&=\frac{1}{2}\left|\scalarproduct{ \A f-g_\delta^n}{ \A f-g_\delta^n}_{\Y}-\scalarproduct{ \A f-g}{ \A f-g}_{\Y}\right|\\			
		&\leq \norm[ \A f]_{\Y}\norm[g_\delta^n-g]_{\Y}+\frac{1}{2}\left|\norm[g_\delta^n]_{\Y}^2-\norm[g]_{\Y}^2\right|
		\end{align*}
		Since $(g_\delta^n)_n \subset \Y$ is bounded and $g\in\Y$, there exists a suitable Lipschitz-constant $L$ such that $\left|\norm[g_\delta^n]_{\Y}^2-\norm[g]_{\Y}^2\right| \leq L \norm[g_\delta^n-g]_{\Y}$ which leads with $\A$ being bounded to the desired result.
	\end{proof}

	With these properties at hand, we can now establish the aforementioned stability regarding $g$.

	\begin{theorem}
		Let $(\delta_n)_n\subseteq\mathbb{R}$ be a positive sequence converging to $0$ and $(g_\delta^n)_n\subseteq \Y,g\in \Y$ such that $\norm[g-g_\delta^n]_{\Y}\leq\delta_n $. Denote with $f_n$ an arbitrary solution to $\min_{u\in L_2(\Omega)} \T_{g_\delta^n,\gamma,\sigma,\alpha}(u)$. Then, every weak accumulation point of $(f_n)_n$ is a minimizer of $\T_{g,\gamma_k,\sigma,\alpha}$ and there exists at least one such point.
	\end{theorem}
	\begin{proof} Since $(g_\delta^n)_n \subset \Y$ converges strongly to $g\in \mathcal{Y}$, we can apply \Cref{prop:consistency}.	
		For arbitrary $\tilde{f}\in L^2(\Omega)$ we have $\mathcal{T}_{g_\delta^n,(\gamma_k)_k,\alpha}(f_n)\leq \mathcal{T}_{g_\delta^n,(\gamma_k)_k,\alpha}(\tilde{f})$ and consequently 
		\begin{equation}
		\begin{split}
		\liminf\limits_{n\to\infty }\mathcal{T}_{g_\delta^n,(\gamma_k)_k,\alpha}(f_n)&\leq\limsup\limits_{n\to\infty }\mathcal{T}_{g_\delta^n,(\gamma_k)_k,\alpha}(f_n)\\
		&\leq\limsup\limits_{n\to\infty }\mathcal{T}_{g_\delta^n,(\gamma_k)_k,\alpha}(\tilde{f}) = \mathcal{T}_{g,(\gamma_k)_k,\alpha}(\tilde{f}),
		\label{eq:upperbound-lim-inf}
		\end{split}
		\end{equation}
		where we used \eqref{eq:functional-consistency} to obtain the last inequality. Due to \cref{eq:uniform-coercivity}, the sequence $(f_n)_n \subset L_2(\Omega)$ is bounded and thus has a weakly convergent subsequence. \par
		Considering such a weakly convergent subsequence $(f_{n_l})_l$ with weak limit $\hat{f}$ as well as a solution $\tilde{f}$  of $\min_{u\in L_2(\Omega)} \mathscr{T}_{g,\gamma,\sigma,\alpha}(u)$ we obtain with the help of the lower semi-continuity of the functional $\T$ in \Cref{lemma:pm-weakly-lower-semicont} 
		\begin{align*}				                
		\T_{g,\gamma,\sigma,\alpha}(\hat{f})
		&\leq \liminf\limits_{l\to\infty }
			\T_{g,\gamma,\sigma,\alpha}(f_{n_l})	\\
			&\leq\liminf\limits_{l\to\infty }\T_{g_\delta^{n_l},\gamma,\sigma,\alpha}(f_{n_l})+\lim\limits_{l\to\infty }\left(\T_{g,\gamma,\sigma,\alpha}(f_{n_l})-\T_{g_\delta^{n_l},\gamma,\sigma,\alpha}(f_{n_l})\right).
		\end{align*}
		The second limit tends to 0 due to \cref{eq:functional-consistency}. Note that $N_\varepsilon$ in \cref{eq:functional-consistency} depends only $\varepsilon$ and the bounding constant $M$. Finally from \cref{eq:upperbound-lim-inf} it follows that $\T_{g,\gamma,\sigma,\alpha}(\hat{f}) \leq \T_{g,\gamma,\sigma,\alpha}(\tilde{f})$ which ends the proof.
	\end{proof}	

	We will now establish convergence for a suitable parameter choice, which shall include that for zero sequences $(\gamma_{1,n})_n,...,(\gamma_{K,n})_n$, $(\alpha_n)_n$, the ratios $\frac{\gamma_{k,n}}{\gamma_{1,n}}=:\tilde{\gamma}_k$ and  $\frac{\alpha_n}{\gamma_{1,n}}=:\tilde{\alpha}$ stay fixed. We will call such sequences \textit{ratio preserving}. This allows us to simplify our penalty term
	\begin{equation*}
			\sum_{k=1}^K \frac{\gamma_{k,n}}{2}\int_{\Omega}p_{k,n}\left(\norm[\nabla_{\sigma_k} f]_2^2\right)\rmd x+\frac{\alpha_n}{2}\norm[f]_{L_2(\Omega)}^2
	\end{equation*}
	to $\tilde{\gamma}_n \mathscr{P}(f)$  with $\tilde{\gamma}_n:=\gamma_{1,n}$ and 
	\begin{equation*}
		\mathscr{P}(f):=\sum_{k=1}^{K}  \frac{\tilde{\gamma}_k}{2}\int_{\Omega}p_k\left(\left\|\nabla_{\sigma_k} f\right\|_2^2\right)+\frac{\tilde{\alpha}}{2}\norm[f]_{L_2(\Omega)}^2.
	\end{equation*}	
	With this notation we can now define more precisely the aforementioned generalization of minimal norm solutions: As in \cite{scherzer09}, we will consider $\P$-minimizing solutions, \textit{i.e.}  solutions of 
	\begin{align}
		\label{eq:reint}
		\min \mathscr{P}(f) \text{ s.t. } \A f=g \text{ for } g\in R(\A).
	\end{align}
	The reinterpretation of our problem $\min \frac12 \|\cdot\|^2 + \mathscr{P}(f)$ as \cref{eq:reint} is justified by the convergence \Cref{thm:convergence}, which mimics the pointwise convergence of regularization methods (towards the generalized inverse). 

	However, in order to prove that theorem, we first have to introduce the notation of level sets,
	\begin{align*}
		\mathcal{L}_{\T_{g_\delta,\gamma,\sigma,\alpha}}(f_0)&:= \left\{ f\in L_2(\Omega)\vert \T_{g_\delta,\gamma,\sigma,\alpha}(f)\leq\T_{g_\delta,\gamma,\sigma,\alpha}(f_0)\right\},\quad f_0\in L_2(\Omega),\\
		\mathcal{L}_{\T_{g_\delta,\gamma,\sigma,\alpha}}(M)&:= \left\{ f\in L_2(\Omega)\vert \T_{g_\delta,\gamma,\sigma,\alpha}(f)\leq M\right\},\quad M\in \mathbb{R}_0^+.
	\end{align*}
	\begin{proposition}
		\label{prop:weakly-seq-closed}
		Let $f_0\in L^2(\Omega)$, $M\in \mathbb{R}_0^+$.
		Then, the level sets $\mathcal{L}_{\T_{g_\delta,\gamma,\sigma,\alpha}}(f_0)$ and $\mathcal{L}_{\T_{g_\delta,\gamma,\sigma,\alpha}}(M)$ are weakly sequentially compact.
	\end{proposition}
	\begin{proof}
		We only consider $\mathcal{L}_{\T_{g_\delta,\gamma,\sigma,\alpha}}(M)$, the other case follows by setting $M:=\T_{g_\delta,\gamma,\sigma,\alpha}(f_0)$.
		The coercivity of $\T_{g_\delta,\gamma,\sigma,\alpha}$ implies that every sequence in $\mathcal{L}_{\T_{g_\delta,\gamma,\sigma,\alpha}}(M)$ is bounded. In particular, every sequence $\mathcal{L}_{\T_{g_\delta,\gamma,\sigma,\alpha}}(M)$ has at least one subsequence which admits a weak limit $f^*$. The lower semicontinuity of our functional now implies $f^*\in\mathcal{L}_{\T_{g_\delta,\gamma,\sigma,\alpha}}(M)$.	
	\end{proof}

    With this property, all items in [\cite{scherzer09}, assumption 3.13] except convexity hold. Hence we can give the following theorem, which is a special case of [\cite{scherzer09}, Theorem 3.26] in which convexity is not required. 

	\begin{theorem}
		\label{thm:convergence}
		Let $f\in L_2(\Omega)$, $g\in \Y$ s.t. $ \A f=g$. Consider sequences  $(f_n)_n\subseteq L_2(\Omega)$, $(g_\delta^n)_n\subseteq \Y$, zero sequences $(\delta_n)_n$, $(\gamma_{k,n})_n$ for $k=1,...,K$, $(\alpha_n)_n\subseteq \mathbb{R}^{+}$ such that
		\begin{enumerate}[label =(\roman*)]
		    \item $\delta_n\to0$ and  $\norm[g-g_\delta^n]_{\Y}\leq \delta_n$ for all $n$,
			\item $(\gamma_{k,n})_n$ and $(\alpha_n)_n$ are ratio preserving in the previous sense,
			\item $\tilde{\gamma}_n:=\gamma_{1,n}$ satisfies $\lim_{n\to\infty} \tilde{\gamma}_{n} = \lim_{n\to\infty} \frac{\delta_n^2}{\tilde{\gamma}_{n}}=0$,
			\item $f_n\in\argmin \T_{g_\delta^n,\gamma_{n},\alpha_n}$.
		\end{enumerate}
		Then the following statements hold:
		\begin{enumerate}[label= \arabic*)]
			\item There is at least one (weak) accumulation point $f_*$ of $(f_n)_n$ and any such $f_*$ solves $\min_u \mathscr{P}(u)\textnormal{ s.t } \A u=g$.
			\item Any weakly convergent subsequence $(f_{n_l})_l$ satisfies
			\begin{equation}
			    \label{eq:reg-term-convergence}
				\lim\limits_{l\to\infty}\mathscr{P}(f_{n_l})=\mathscr{P}(f_*),
			\end{equation}
			where $f_*$ is an arbitrary accumulation point of $(f_n)_n$.
			\item If the solution $f_*$ of $\min_{u} \mathscr{P}(u)\textnormal{ s.t. } \A u=g$ is unique, then $f_n\rightharpoonup f_*$.
		\end{enumerate}
	\end{theorem}

	Until now, we have only considered weak convergence. This stems from the fact that in infinite dimensions, even bounded sequences do not have necessarily (strongly) convergent subsequences. Of course one could always assume, that the sequences converge strongly, but this only reduces the generality of the framework. In this context, our $\norm[\cdot]_{L^2(\Omega)}$ introduced to enforce coercivity has an additional side effect regarding the last Theorem. In fact, we actually can replace the weak convergence of subsequences by strong convergence, as the following corollary shows.
	\begin{corollary}
		The weakly convergent subsequences $(f_{n_l})_l$ in \Cref{thm:convergence} are also strongly convergent.
	\end{corollary}
	\begin{proof}
		Let $f_*$ s.t. $f_{n_l}\rightharpoonup f_*$.
		Then \cref{eq:strong-conv} and \cref{eq:reg-term-convergence} imply $\norm[f_{n_l}]_{L_2(\Omega)}\to\norm[f_*]_{L_2(\Omega)}$. 
	\end{proof}
	In this section we have shown that we can get stable solutions associated to the ill-posed problem $\mathscr{A}f=g$ by minimizing a general NID-functional \eqref{eq:def_NID}. 
	The construction of such solutions by gradient descent is worked out in the following section.

\section{A gradient descent method for the NID model}
	\label{subsec:gradient-descent-for-pm-reg}
	In this section we will propose and discuss an algorithm to solve the optimization problem (\ref{eq:def_NID}) and examine its convergence. We first show that under certain conditions it is suitable for gradient methods.
	Similar to the finite dimensional case, we must ensure continuity conditions. 
	
	\begin{theorem}		
		\label{thm:continuity-derivative}
		Let $g\in \Y$, $\sigma_k\in\mathbb{R}^+$, $\gamma_k\in\mathbb{R}^+_0$ for $k=1,\dots,K$,		$\alpha\in\mathbb{R}$ and consider let $\T_{g,\gamma,\sigma,\alpha}$
		be as in \cref{eq:def_NID}. If, in addition to the previous assumptions, the penalty terms $p_k \in C_{+,b}^3(\mathbb{R}_0^+)$, $k=1,\dots,K$, then the following properties hold: 
		\begin{enumerate}[label= \arabic*)]
			\item the Fr\'echet derivative of $\T_{g,\gamma,\sigma,\alpha}$ is given by
			\begin{equation}
				\label{eq:frech-deriv-func}
				\T_{g,\gamma,\sigma,\alpha}'(f)=
				\A^*( \A f-g)-\sum_{k=1}^{K}\gamma_k(\nabla_{\sigma_k})^*\left(\varphi_k\left(\norm[\nabla_{\sigma_k} f]_2^2\right)\nabla_{\sigma_k} f\right)+\alpha f.
			\end{equation}
			\item $\T_{g,\gamma,\sigma,\alpha}'$ is Lipschitz continuous, i.e.  for all  $f_1,f_2\in L_2(\Omega)$ there exists some $\kappa>0$ such that 
		\begin{equation}
				\label{eq:frechet-lipschitz}
				\norm[\T_{g,\gamma,\sigma,\alpha}'(f_1)-\T_{g,\gamma,\sigma,\alpha}'(f_2)]_{\mathcal{L}(L_2(\Omega),\mathbb{R})}\leq \kappa\norm[f_1-f_2]_{L_2(\Omega)}.
			\end{equation}
			\item Let $f_n \rightharpoonup f$, then the functions $\left(\T_{g,\gamma,\sigma,\alpha}'(f_n)\right)_n\subseteq L_2(\Omega),\T_{g,\gamma,\sigma,\alpha}'(f)\in L_2(\Omega)$ satisfy:
			\begin{equation}
			    \label{eq:weak-weak-continuity}
                \T_{g,\gamma,\sigma,\alpha}'(f_n) \rightharpoonup \T_{g,\gamma,\sigma,\alpha}'(f)
			\end{equation}		
		\end{enumerate}
	\end{theorem}
	\begin{proof}
	    See \Cref{proof:continuity-derivative}.
	\end{proof}
	\begin{corollary}
		\label{corollary:lipschitz-cont}
		The mappings $\T_{g,\gamma,\sigma,\alpha}:L_2(\Omega)\to\mathbb{R}$ and $\T'_{g,\gamma,\sigma,\alpha}:L_2(\Omega)\to \mathcal{L}\left(L_2(\Omega),\mathbb{R}\right)$, defined in \Cref{thm:continuity-derivative}, map bounded sets onto bounded sets. In particular, $\T_{g,\gamma,\sigma,\alpha}$ is Lipschitz continuous on bounded sets.
	\end{corollary}
	\begin{proof}
		The assertion regarding $\T'_{g,\gamma,\sigma,\alpha}$ follows directly from the Lipschitz continuity of the Fr\'echet derivative. 
    From the mean value theorem we obtain for $f_1,f_2\in L_2(\Omega)$ the inequality
\begin{align}
    \label{eq:mean-val}
    &\left|\T_{g,\gamma,\sigma,\alpha}(f_2)-\T_{g,\gamma,\sigma,\alpha}(f_1)\right|\notag\\\leq &\sup_{t\in[0,1]}\norm[\T'_{g,\gamma,\sigma,\alpha}(f_1+t(f_2-f_1))]_{\mathcal{L}(L_2(\Omega),\mathbb{R})}\norm[f_2-f_1]_{L^2(\Omega)}.
\end{align}
Now let $f_1,f_2$ be arbitrary elements of a subset of $L_2(\Omega)$ bounded by some constant $M>0$. Since $t\in[0,1]$, we obtain
\begin{align*}
    \norm[f_1+t(f_2-f_1)]_{L_2(\Omega)}&\leq \norm[f_1]_{L_2(\Omega)}+t\norm[f_2-f_1]_{L_2(\Omega)}\\
    &\leq 2\norm[f_1]_{L_2(\Omega)}+\norm[f_2]_{L_2(\Omega)}\leq 3M,
\end{align*}
which combined with the boundedness of $\T'_{g,\gamma,\sigma,\alpha}$ and \cref{eq:mean-val} implies Lipschitz continuity of $\T_{g,\gamma,\sigma,\alpha}$ on bounded sets and therefore boundedness of $\T_{g,\gamma,\sigma,\alpha}$ on such sets. 

	\end{proof}
	Using basic properties of the Fr\'{e}chet-derivative 
	as well as $-\divergence =\nabla^*$ (see \cite[Theorem 6.88]{bredies11}), we can formally identify the gradient
	of our functional with
	\begin{equation*}
		\A^*( \A f-g)-\sum_{k=1}^K\gamma_k G_{\sigma_k}^*\divergence\left(\varphi_k\left(\norm[\nabla_{\sigma_k} f]_2^2\right)\nabla_{\sigma_k} f\right)+\alpha f.
	\end{equation*}
	Now, similarly to the finite dimensional case, we can use a gradient descent method in order to minimize $		\T_{g,\gamma,\sigma,\alpha}(f)$ via an iteration
	\begin{equation*}
		f_{n+1} = f_n-t_n v_n,\qquad t_n>0
	\end{equation*}
	where $t_n >0$ is a suitable step size for all $n$ and $v_n= \T'_{g,\gamma,\sigma,\alpha}(f_n)$.	

		Due to the smoothness of our functional	we will adapt the step size given in \cite{barzilai88}: The authors derived their choice  from the secant equation inherited in quasi Newton methods by minimizing $r(\tau):= \norm[f_n-f_{n-1}-\tau(v_n-v_{n-1})]_2^2$. Here, we propose to use the corresponding step size given by
	\begin{equation*}
		\tau_n := \frac{\scalarproduct{v_n-v_{n-1}}{f_{n}-f_{n-1}}_{L_2(\Omega)}}{\scalarproduct{v_n-v_{n-1}}{v_{n}-v_{n-1}}_{L_2(\Omega)}}.
	\end{equation*}
	We will later prove that this gives indeed a valid step size.\\
	In this case we can use a modified \textbf{Armijo rule} by computing \begin{equation}
	     \label{eq:mod-arm}
		t_n =\max\left\{\beta^l \tau_n \, \vert \, l=0,1,2,...,\right\}
	\end{equation}
	satisfying
	\begin{equation}
    	\label{eq:wolfe-1} 
		\T_{g,\gamma,\sigma,\alpha}(f_n-t_n v_n)\leq \T_{g,\gamma,\sigma,\alpha}(f_n)-\mu t_n 
		\scalarproduct{v_n}{v_n}_{L_2(\Omega)}
	\end{equation}
	with $\mu,\beta\in (0,1)$. Another option is to compute $t_n$ from $\tau_n$ such that \textbf{Wolfe conditions} hold, \textit{i.e.}, in addition to \cref{eq:wolfe-1}, the inequality
	\begin{equation}
		\label{eq:wolfe-2} 
		-\scalarproduct{\T_{g,\gamma,\sigma,\alpha}'(f_n-t_n v_n)}{v_n}_{L_2(\Omega)} \geq -\varrho 		\scalarproduct{v_n}{v_n}_{L_2(\Omega)}
	\end{equation}
	with $0<\mu<\varrho<1$ is satisfied. A suitable  algorithm to do this can be found in \cite{izmailov14}. Since $\T_{g,\gamma,\sigma,\alpha}$ is Fr\'echet-differentiable,  coercive and bounded from below, one can apply the same arguments as in the finite dimensional case (see \cite{izmailov14}, Lemma 2.22 and its proof) to guarantee that the computation stops  after a finite number of iterations.\\\indent
    While the Armijo rule is easier to implement and only needs an update of the functional value during each update of $t_n$, more precautions have to be taken in order to guarantee a sufficient decrease in the objective functional as well as well-definedness of the $\tau_n$ and (global) convergence of the method, \textit{e.g.} by taking a similar approach as \cite{grippo86} or \cite{raydan97}. The computation of a step size satisfying the Wolfe conditions on the other hand is more expensive per iteration, needing also an update of the gradient. Its advantage, however, is that it lacks all of the mentioned problems with the Armijo rule. This leads to the following algorithm.
	\begin{algorithm}
		\caption{Compute minimizer of $\T_{g,\gamma,\sigma,\alpha}(f)$}
		\label{alg:alg-final}
		\textbf{Input}: $g$, $\gamma_{k}>0$,$\sigma_k>0$,$\varphi_k$, $\alpha>0$, $f_0$, $\tau_0>0$, \textbf{Choose} $0<\mu<\frac{1}{2},\mu<\varrho<1$\\
		\textbf{Output}: (Approximate) solution of $\min\T_{g,\gamma,\sigma,\alpha}(f)$
		\begin{algorithmic}
			\FORALL{$n=0,1,2,...$}
			\STATE $v_n\leftarrow \A^*( \A f_n-g)-\sum_{k=1}^K\gamma_k G_{\sigma_k}^*\divergence\left(\varphi_k\left(\norm[\nabla_{\sigma_k} f_n]_2^2\right)\nabla_{\sigma_k} f_n\right)+\alpha f_n$
			\IF {$\norm[v_n]=0$}
			\RETURN $f_n$
			\ENDIF
			\IF {$n>0$}
			\STATE $\tau_{n}\leftarrow\frac{\scalarproduct{f_{n}-f_{n-1}}{v_{n}-v_{n-1}}}{\norm[v_{n}-v_{n-1}]^2}$
			\ENDIF 
			\STATE Compute $t_n$ according to eqs. (\ref{eq:mod-arm}), (\ref{eq:wolfe-1}) and (\ref{eq:wolfe-2})
			\STATE $f_{n+1}\leftarrow f_{n}-t_n v_n$
			\STATE $n\leftarrow n+1$
			\ENDFOR
		\end{algorithmic}
	\end{algorithm}

This algorithm produces a sequence $(f_n)_n$, such that convergent subsequences converge weakly to stationary points. To show this, we first have to ensure that $\tau_n$ is indeed well defined. To this aim, we will assume w. l. o. g. that every generated $v_n\neq 0$. Otherwise the algorithm would stop.
	\begin{lemma}
		For $n>0$, let $f_n,v_n,f_{n-1},v_{n-1}$ be given by Algorithm \ref{alg:alg-final}. Then $\tau_{n}$ is well-defined and positive.
	\end{lemma}
	\begin{proof} Eq. 
		\eqref{eq:wolfe-2} holds in the $(n-1)$th iteration step, \textit{i.e.} 
		\begin{equation}
			\label{eq:strict-ineq}
			-\scalarproduct{v_{n}}{v_{n-1}}\geq -\varrho\norm[v_{n-1}]^2>-\norm[v_{n-1}]^2.
		\end{equation}
		Since \cref{eq:strict-ineq} is equivalent to $\scalarproduct{v_{n}-v_{n-1}}{v_{n-1}}<0$, this implies $v_n\neq v_{n-1}$ and thus 
		\begin{equation*}
			\tau_{n}=\frac{\scalarproduct{f_n-f_{n-1}}{v_{n}-v_{n-1}}}{\scalarproduct{v_n-v_{n-1}}{v_{n}-v_{n-1}}}=\frac{-t_{n-1}\scalarproduct{v_{n-1}}{v_{n}-v_{n-1}}}{\norm[v_{n}-v_{n-1}]^2} \geq 0.
		\end{equation*}		
	\end{proof}

	Assuming the algorithm never stops, \textit{i.e.} $v_n\neq 0$ for all $n$, one can now prove the following (weak) convergence theorem for \Cref{alg:alg-final} which is similar to the theorems given in \cite{grippo86} and \cite{raydan97} for the finite dimensional case.
	\begin{theorem}\label{thm:convergence-alg}
		Let $(f_n)_n$ be a sequence generated by \Cref{alg:alg-final} and let  $p_k\in C^3_{+,b}\left(\mathbb{R}_0^+\right)$ for $k=1,\dots,K$. Then, the following properties hold:
		\begin{enumerate}[label= \arabic*)]
			\item There is at least one weak accumulation point.
			\item There is a fixed $\theta>0$ s.t. 
				\begin{align}
					\label{eq:efficiency}					
					\T_{g,\gamma,\sigma,\alpha}(f_n-t_n v_n)\leq \T_{g,\gamma,\sigma,\alpha}(f_n)-\theta \norm[v_n]^2.
				\end{align}
			\item \label{item:conerverge-alg-item} It holds $\lim_{n\to\infty}\norm[v_n]\to 0$.
			\item Every weak accumulation point $f_*$ is a stationary point.
			\item Every subsequence $(f_{n_l})_l$ converging (weakly) to $f_*$ satisfies
			\begin{equation*}				        
			   \T_{g,\gamma,\sigma,\alpha}(f_*)\leq \liminf_{l\to \infty}\T_{g,\gamma,\sigma,\alpha}(f_{n_l}).
			\end{equation*}
			\item If $f_{n_l}\to f$ strongly for a subsequence, then $f$ is not a maximum.
		\end{enumerate}
	\end{theorem}
	\begin{proof}\mbox{}\\[-0.9em]
		\begin{enumerate}[label= \arabic*)]
			\item Since the sequence $(f_n)_n$ is generated by \ref{alg:alg-final}, the condition (\ref{eq:wolfe-1}) is satisfied and enforces $(\T_{g,\gamma,\sigma,\alpha}(f_n))_n$ to be non-increasing. Therefore, $f_n\in\mathcal{L}_{\T_{g,\gamma,\sigma,\alpha}}(f_0)\quad\text{for all } n$ and \Cref{prop:weakly-seq-closed} delivers the assertion.
			\item 
			Due to the Lipschitz continuity of the Fr\'echet derivative in  \cref{eq:frechet-lipschitz} and the Wolfe condition in  \cref{eq:wolfe-2}, it holds 
			\begin{align}\label{eq:bounded_tn}
			\left(1-\varrho\right)\norm[v_n]^2\leq \scalarproduct{v_n-v_{n+1}}{v_n}\leq \kappa t_n\norm[v_n]^2
			\quad \Rightarrow \quad 
			\frac{1-\varrho}{\kappa}\leq t_n.
			\end{align}
			\Cref{eq:efficiency} finally follows on from \cref{eq:wolfe-1} by choosing $\theta=\frac{(1-\varrho)\mu}{\kappa}$.
			
			\item Similarly to [\cite{grippo86}, p. 710, equation (7)], we first show that
$$\lim_{n\to\infty}t_{n}\norm[v_{n}]^2_{L^2(\Omega)}=0. $$
			By \cref{eq:wolfe-1}, it holds
			\begin{align*}
		        \mu t_n 
		        \scalarproduct{v_n}{v_n}_{L_2(\Omega)}\leq \T_{g,\gamma,\sigma,\alpha}(f_{n})- \T_{g,\gamma,\sigma,\alpha}(f_{n+1}).
			\end{align*}
			Since $(\T_{g,\gamma,\sigma,\alpha}(f_n))_n$ is a monotonically decreasing sequence bounded from below, the right hand side tends to 0 implying that $t_{n}\norm[v_{n}]^2_{L^2(\Omega)} \to 0$. Since $t_n$ is bounded from below, see \cref{eq:bounded_tn}, the desired result follows.
			\item Let $v^*=\T_{g,\gamma,\sigma,\alpha}(f^*)$. By \cref{eq:weak-weak-continuity}, we have $\T_{g,\gamma,\sigma,\alpha}(f_{n_l})\rightharpoonup\T_{g,\gamma,\sigma,\alpha}(f^*)$ for every sequence $(f_{n_l})_l$ converging weakly to $f^*$. The lower semicontinuity of the norm and \ref{item:conerverge-alg-item} finish the proof.
			\item This is a consequence of \Cref{lemma:pm-weakly-lower-semicont}.
			\item 
			Analogously to the finite dimensional case in \cite{grippo86}, the assertion follows from the fact, that  $(\T_{g,\gamma,\sigma,\alpha}(f_n))_n$ is strictly decreasing.
		\end{enumerate}
	\end{proof}
	Property (\ref{eq:efficiency}) describes the previously mentioned \glqq{}efficiency\grqq{} as the descent in the functional value is proportional to the gradient norm. 
	
	\par
	In contrast with \cite{grippo86} or \cite{raydan97}, we need to distinguish between 5) and 6) as in infinite dimensions, bounded sequences can have no convergent subsequences. Furthermore, the strong convergence results from \Cref{subsec:well-posedness-of-pm-reg} are not applicable either since they rely on the strong convergence of the penalty term which is only guaranteed if the noise level and the corresponding parameter choice go to zero.
	
	\par
	In practice, we only have discrete measurements available. However, using suitable discretizations (see \cite[Lemma 6.142]{bredies11}) for the scalar products and the differential operators, the discretization of the continuous gradient method coincides with the gradient method for the discretized model. This is an interesting property of the proposed continuous approach which can be straightforwardly applied to discrete problems.
	
\par Based on NID regularizer studied in \Cref{subsec:well-posedness-of-pm-reg} and \Cref{subsec:gradient-descent-for-pm-reg}, we propose to generalize this idea allowing variations of the penalty terms in order to better reflect the evolution of the iterated solution.

\section{Construction of an adaptive NID regularizer}	\label{sec:a-nid}

	One problem with the previous model is, that there might not be sufficient information on the edges in the first few iterates for the diffusion process to use its full potential. In \cite{li94} the authors considered a standard Perona-Malik functional with changing thresholds. We propose a more general adaptive approach, switching from a general TV regularization to a NID regularizer as described above.	More precisely, the general idea is to consider a mixture of both regularization terms in the gradient descent iterate as following:
	\begin{align*}
		v_n &= \A^*( \A f_n-g)+\alpha f_n -\omega(n)\beta \divergence\left(\frac{\nabla f}{\sqrt{\varepsilon+\norm[\nabla f]_2^2}}\right)\\
		&\qquad -(1-\omega(n))\left(\sum_{k=1}^K\gamma_k G_{\sigma_k}^*\divergence\left(\varphi_{k,n}\left(\norm[\nabla_{\sigma_k} f_n]^2_2\right)\nabla_{\sigma_k} f_n\right)\right)\\
		 f_{n+1}&=f_n-t_n v_n,
	\end{align*}
with $t_n$ a given step size. In the above construction, the NID term depends on the iteration $n$ as indicated by $\varphi_{k,n}$. This evolution of the NID term can be, for instance, done by changing weights and thresholds. This construction of the gradient descent can be understood as solving a changing minimization problem $\min\limits_f \T_n(f)$ with 
\begin{align*}
		\T_n(f)&:=\frac{1}{2}\norm[ \A f-g]_\Y^2
		+\frac{\alpha}{2}\norm[f]^2_{L_2(\Omega)}
		+\omega(n)\frac{\beta}{2}\int_{\Omega}\sqrt{\varepsilon+\norm[\nabla f]_2^2} \rmd x \\
	&\qquad +	(1-\omega(n))\left(\sum_{k=1}^K\frac{\gamma_k}{2}\int_{\Omega}p_k\left(\norm[\nabla_{\sigma_k} f]^2_2\right)\rmd x \right),
	\end{align*}	
	where $\omega:\mathbb{N}_0\to \left[0,1\right]$ is a monotonically decreasing function. \par
	Similarly to \Cref{thm:continuity-derivative} and \Cref{corollary:lipschitz-cont}, we can establish the following properties. 
	\begin{theorem}
		\label{thm:cont-der-mixed}
		For every $n\in\mathbb{N}_0$, it holds:
		\begin{enumerate}[label= \arabic*)]
			\item 
			The mapping $\T_n:\Sob^1(\Omega)\to \mathbb{R}$
			is Fr\'{e}chet-differentiable with Fr\'{e}chet-derivative 		
			\begin{align*}
				\T_n'(f)&=\A^*( \A f-g) -\omega(n)\beta \divergence\left(\frac{1}{\sqrt{\varepsilon+\norm[\nabla f]_2^2}}\nabla f\right)	\\
				&\qquad-(1-\omega(n))\sum_{k=1}^{K}\gamma \nabla_\sigma^*\divergence\left(\varphi_{n,k}\left(\norm[\nabla_{\sigma_k} f_n]^2_2\right)\nabla_{\sigma_k} f_n\right)+\alpha f_n.
			\end{align*}

			\item $\T_n'$ is Lipschitz in the sense s.t. 
			\begin{equation*}
			\norm[\T_n'(f_1)-\T_n'(f_2)]\leq \kappa_n\norm[f_1-f_2]_{L_2(\Omega)}\quad\text{for all } f_1,f_2\in L_2(\Omega)
			\end{equation*}
			with some $\kappa_n>0$.
			\item The mappings	$\T_n:\Sob^1(\Omega)\to \mathbb{R}$ and $\T_n':\Sob^1(\Omega)\to\mathcal{L}\left(L_2(\Omega),\mathbb{R}\right)$ are bounded on bounded sets. In particular, the former is Lipschitz continuous on bounded sets.	
		\end{enumerate}
	\end{theorem}
	\begin{proof}
		We can decompose our functional $\T_n$ in 
		\begin{align*}
			\T_n(f)&=(1-\omega(n))\T_{NID}^n(f)+\omega(n)\T_{TV}(f)
			\intertext{with}
			 \T_{NID}^n(f)&=\frac{1}{2}\norm[ \A f-g]_{L_2(\tilde{\Omega})}+\sum_{k=1}^{K}\int_{\Omega}p_{k,n}(\norm[\nabla_{\sigma_k} f]^2_2)\,\rmd x+\frac{\alpha}{2}\norm[f]^2_{L_2(\Omega)}\\			
			 \T_{TV}(f)&=\frac{1}{2}\norm[ \A f-g]_{L_2(\tilde{\Omega})}+\frac{\beta}{2}\int_{\Omega}\sqrt{\varepsilon+\norm[\nabla  f]_2^2}\,\rmd x+\frac{\alpha}{2}\norm[f]^2_{L_2(\Omega)}
		\end{align*}
		\begin{enumerate}[label= \arabic*)]
		\item Since $f\in \Sob^1(\Omega)$ and $\nabla:\Sob^1(\Omega)\to L_2(\Omega)$ is bounded, it is in particular Fr\'{e}chet-differentiable. We thus can apply the proof of \Cref{thm:continuity-derivative} directly to $\T_{TV}$ using the penalty function $p_{AV}(s)=\sqrt{\varepsilon+s}$. 
		Since we can further apply \Cref{thm:continuity-derivative} to $\T_{NID}$, the result follows directly by linearity. 
		\item Analogously we can deduce by linearity
		\begin{equation*}
			\norm[\T_n'(f_1)-\T_n'(f_2)]\leq \kappa_n\norm[f_1-f_2]_{L_2(\Omega)}\quad \text{for all }  f_1,f_2\in L_2(\Omega)
		\end{equation*}
		for some $\kappa_n>0$.
		\item The Lipschitz continuity follows using the summation rule and the argument from the proof of \Cref{corollary:lipschitz-cont}. 
	    \end{enumerate}
	\end{proof}
	This theorem provides a sufficient regularity of the A-NID functional in order to give a gradient descent type algorithm.	A suitable step size can be computed using the Armijo rule, \textit{i.e.}  $t_n:=\tau^l$, where $0<\tau<1$ is fixed and $l$ is the smallest number in every step s.t. 
	\begin{equation}
	\label{eq:mixed_step_cond}
	\T_n(f_n-\tau^l v_n)\leq \T_n(f_n)-\mu \tau^l \scalarproduct{v_n}{v_n}_{L_2(\Omega)}
	\end{equation}
	holds, leading to Algorithm \ref{alg:alg-mixed}. 	
	\begin{algorithm} 
		\caption{Adaptive NID}
		\label{alg:alg-mixed}
		\textbf{Input}: $g$, $(\gamma_{k})_k>0$, $\alpha>0$, $f_0$, $\tau_0>0$, \textbf{Choose} $0<\mu<\frac{1}{2},\mu<\varrho<1$\\
		\textbf{Output}: (Approximate) solution of $\T^1(f)$
		\begin{algorithmic}
			\FORALL{$n=0,1,2,...$}
			\STATE $v_n\leftarrow (1-\omega(n)) {\T_{NID}^n}' (f_n)+\omega(n){\T_{TV}}'(f_n)$
			\IF {$\norm[v_n]=0$}
			\RETURN $f_n$
			\ENDIF
			\STATE Compute $t_n$ satisfying \cref{eq:mixed_step_cond} 
			\STATE $f_{n+1}\leftarrow f_{n}-t_n v_n$
			\STATE $n\leftarrow n+1$
			\ENDFOR
		\end{algorithmic}
	\end{algorithm}
	\par 
	In order to analyze convergence to a steady state, we impose the following condition on $\omega$ and $(\T_{NID}^n)_n$: Let there be an $N>0$ s.t. for all $n>N$
	\begin{align}\label{eq:condition_nid_nk_2}
	\mu t_n\norm[v_n]^2 \geq &\left(\omega(n)-\omega(n+1)\right) \left(\T_{NID}^{n}(f_{n+1}) - \T_{TV}(f_{n+1}) \right) \nonumber\\ 
	&+ (1-\omega(n+1)) \left(\T_{NID}^{n+1}(f_{n+1})  - \T_{NID}^{n}(f_{n+1}) \right) .
	\end{align}
This relation gives a control of the variations w.r.t. $n$ allowed for $\T_{NID}^{n}$ and $\omega(n)$ such that \Cref{alg:alg-mixed} produces for all $n>N$ a proper gradient descent.	Assuming as before $v_n\neq 0$ for all generated $f_n$, we can then state the following theorem.
	\begin{theorem}
		\label{thm:conv-anid}
		Let $(\T_{NID}^n)_n$ be uniformly coercive with respect to $\norm[\cdot]_{\Sob^1(\Omega)}$, \textit{i.e.} $\T_{NID}^j(f_j)\to\infty$ for any sequence $(f_j)_j\subset \Sob^1(\Omega)$ with $\norm[f_j]_{\Sob^1(\Omega)}\to\infty$. Let furthermore
		\begin{itemize}
			\item $(p_{k,n})_n$ s.t. there exist constants $L_k$ satisfying $L_k\geq \sup_{s\in\mathbb{R}} p_{k,n}''(s)$ for all $n$; 			
			\item $(f_n)_n$ be a sequence generated by \Cref{alg:alg-mixed};
			\item $(T_{NID}^n)_n$ and $\omega$ s.t. there is an $N$ for which \cref{eq:condition_nid_nk_2} holds;
			\item $\varepsilon>0$ s.t. there exists $N_\varepsilon$ with $\omega(n)\geq \varepsilon$ for all $n> N_\varepsilon$.
		\end{itemize}
		Then the sequence $(f_n)_n$ satisfies the following properties, similar to \Cref{thm:convergence-alg}:
		\begin{enumerate}[label= \arabic*)]
			\item If $\norm[v_n]\neq 0$, $t_n$ is well defined and there is a fixed $\theta>0$ (independent from $n$) s.t. $t_n\geq \theta$. In particular $f_{n+1}$ is well defined.
			\item $\left(\T_n(f_n)\right)_{n>N}$ is (strictly) decreasing, in particular it stays bounded.
			\item $(f_n)_n$ stays bounded, in particular we have at least one weak accumulation point.
			\item It holds that $\lim_{n\to\infty}\norm[v_n]\to 0$.
		\end{enumerate}
	\end{theorem}
	\begin{proof} See \Cref{proof:conv-anid}.
	\end{proof}

We have established convergence of the gradient descent approach for the adaptive NID functional defined in \cref{eq:def_aNID}. We illustrate the proposed approach at an example from computerized tomography and compare its performance to the standard NID approach (\ref{eq:def_NID}) and the standard TV regularization in the following section.

	\section{Application to computerized tomography}
	\label{sec:numerical-implementation}

\subsection{Recalls on the Radon transform}
\par
In order to illustrate and validate the proposed regularization technique, we consider the image reconstruction problem in computerized tomography (CT). In this case, the forward model is described by the well-known Radon transform defined by
\begin{equation*}
g(\theta,s)= \A f(\theta,s):=\int_{\R} f(s\theta + q \theta^\perp)\rmd q \qquad \theta \in S^{1},s\in\mathbb{R}.
\end{equation*}
A regularized solution to the inverse problem $g = \A f$ can be obtained by the filtered backprojection (FBP) algorithm which can be expressed using the 1D-Fourier transform $\F$ acting on the component $s$ by
\begin{equation*}
f(x)=\frac{1}{4\pi^2} \A^*\F^{-1}(F_\gamma (\sigma)\left|\sigma\right|\F {g}(\theta,\sigma) )(x)
\end{equation*}
with $\A^*$ the adjoint of $\A$ and $F_\gamma$ a low pass filter. For the results here we used the standard Shepp-Logan filter
\begin{equation}\label{eq:shepp-logan}
	F_\gamma (\sigma)=\left\{\begin{array}{cl} \sinc\left(\frac{\pi\sigma}{2\gamma}\right) & \left|\sigma\right|\leq \gamma\\0& \textnormal{ otherwise.}
\end{array}\right.
\end{equation}
For a more detailed discussion on CT, we refer the reader to \cite{natterer01}. In the simulations, we consider the well-known Shepp-Logan phantom (see  \Cref{fig:phantom}). The corresponding projections or \textit{sinogram} is given in \Cref{fig:sino} with and without noise. We used an additive normally distributed noise which leads to the following signal-to-noise ratios: $\snr\approx 19.4$dB and $\psnr\approx 35.0$dB.

\begin{figure}[t]\centering
\begin{subfigure}[c]{0.49\linewidth}
\includegraphics[width=\linewidth]{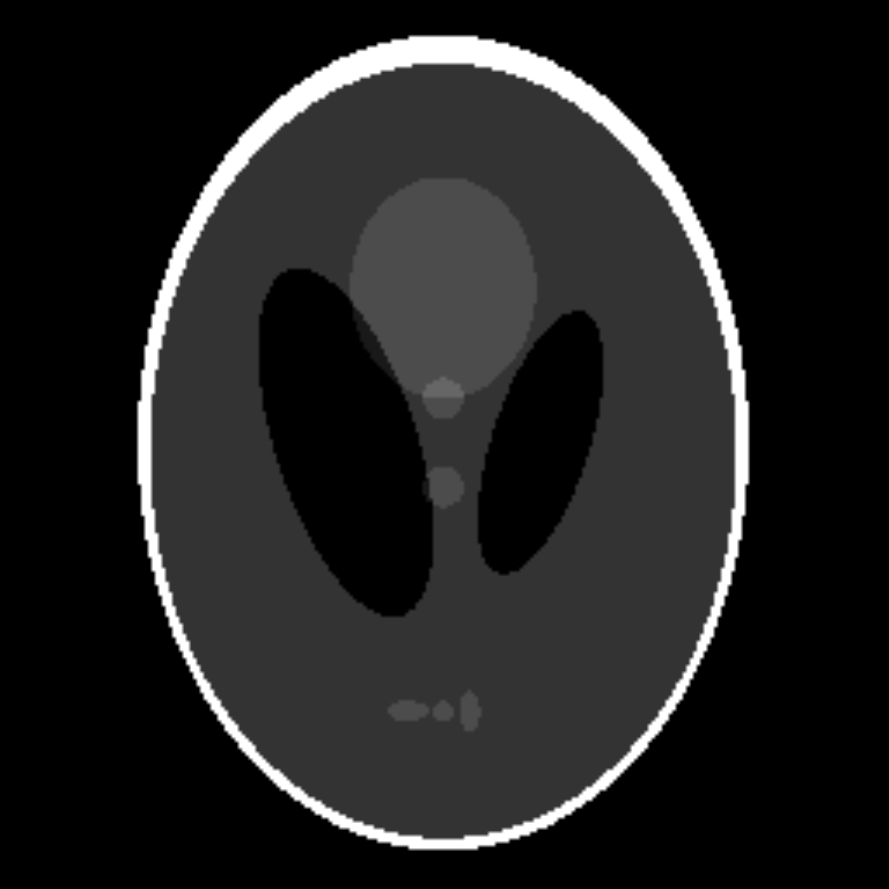}
    \caption{}
    \label{fig:phantom}
\end{subfigure}
\begin{subfigure}[c]{0.49\linewidth}
    \includegraphics[width=\linewidth]{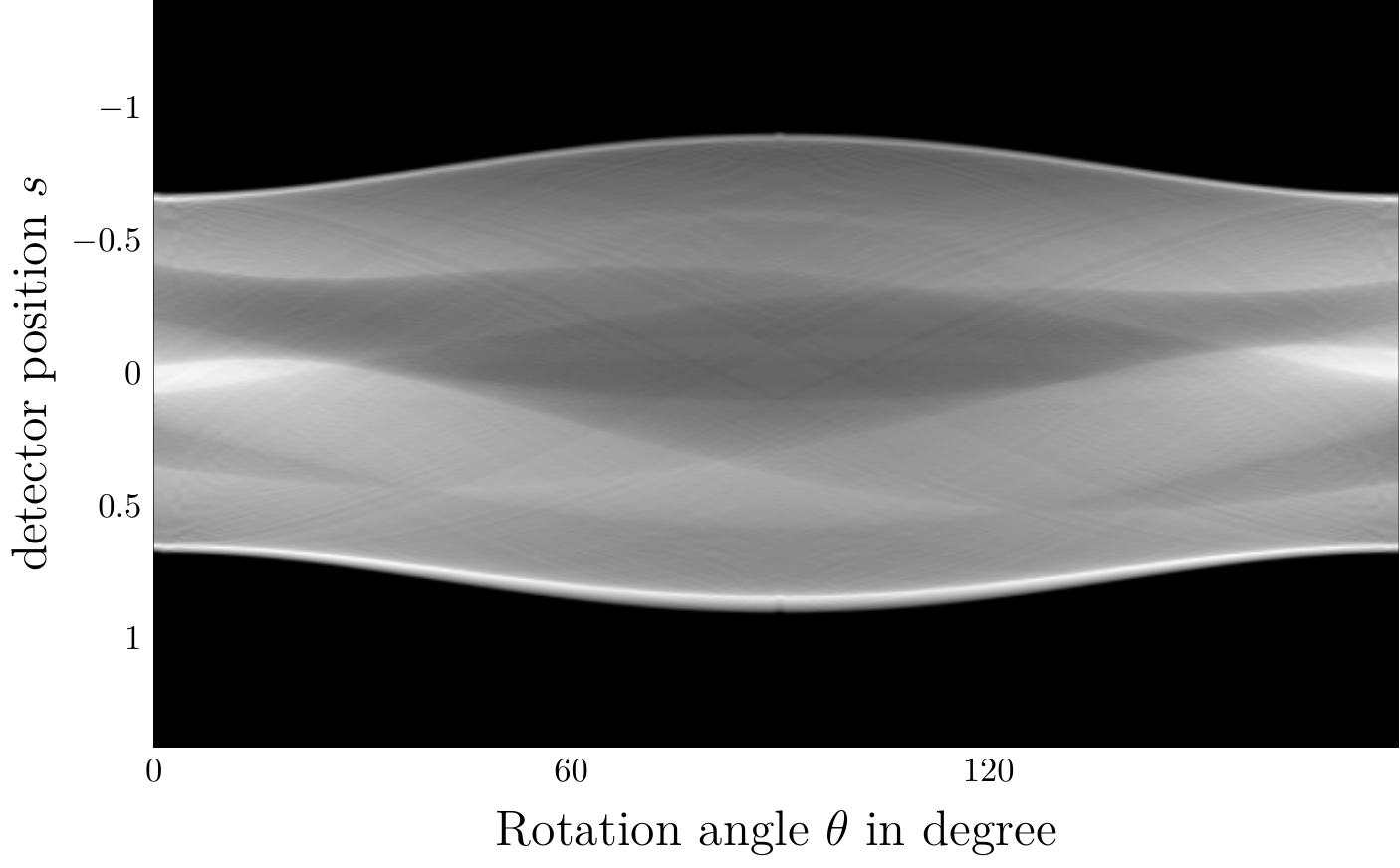}\\
     \includegraphics[width=\linewidth]{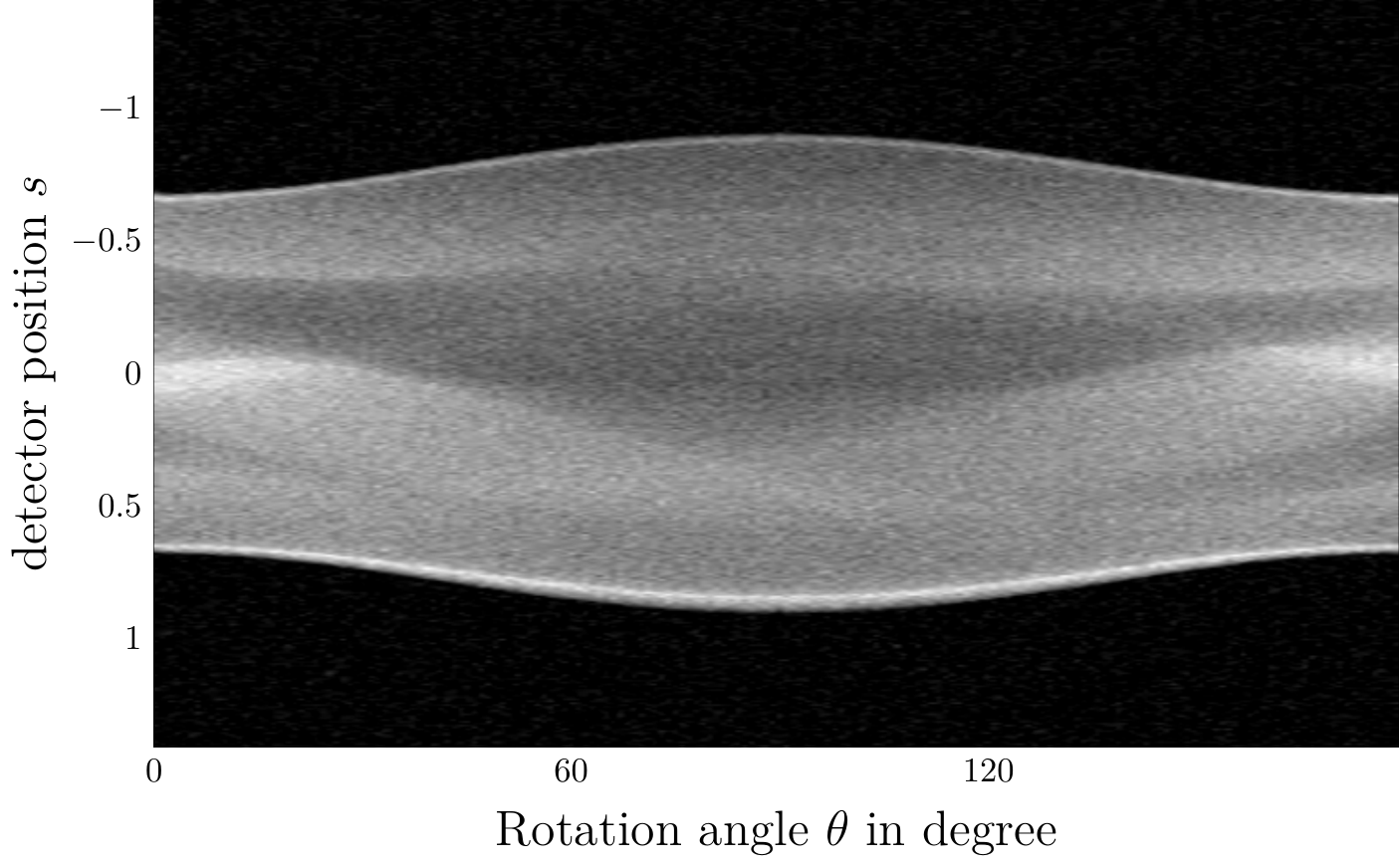}
    \caption{}
    \label{fig:sino}
\end{subfigure}
\caption{(a) Shepp-Logan phantom; (b) Exact and noisy sinograms.}
\end{figure}

\subsection{Discretization}	

In order to implement our approach on realistic data, we must consider a proper discretization. The domain of the image $[-1,1]\times [-1,1]$ is sampled on a standard grid of $N^2$ pixels with inner width $h=\frac{2}{N}$. As depicted in \cref{fig:sino}, the Radon transform data is sampled with $Q$ detector positions denoted by $(s_l)_{l=1,\ldots,Q}$ and $P$ angles ranging from 0 to $\pi$ and denoted by $(\theta_j)_{j=1,\ldots,P}$. More suited for iterative scheme, the Radon transform is approximated by a projection matrix $\in \R^{P Q\times N^2}$ whose entries return  the length of the intersection between a pixel $(m,n)$ and a straight line $(\theta_j,s_l)$. (The values used where $N=256,P=Q=180$.)\\

\par	
For the corresponding discretization of the differential operators we used an isotropic forward-backward-differences scheme 
    \begin{align*}	    
    	\divergence(\varphi(\norm[\nabla f]^2_2)\nabla f)(x_{i,j})&\approx \frac{w_{i,j}}{h}\left((\partial_1 f)_{i,j}-(\partial_1 f)_{i,j-1}\right)\\
	    &\quad +\frac{w_{i,j}}{h}\left((\partial_2 f)_{i,j}-(\partial_2 f)_{i+1,j}\right).
    \end{align*}
    with 
    \begin{align*}
    	w_{i,j}=\varphi\left(\norm[\nabla f(x_{i,j})]^2_2\right)&\approx \varphi\left(\left|(\partial_1 f)_{i,j}\right|^2+\left|(\partial_2 f)_{i,j}\right|^2\right)
    \end{align*}
    and
    \begin{align*}
    	\left(\partial_{1}f\right)_{i,j}&:=\left\{\begin{array}{cl} \frac{f_{i,j+1}-f_{i,j}}{h}, & j<N \\ 0 & j=N\end{array}\right.\approx \frac{\partial}{\partial x_1} f(x_{i,j})\\
    	\left(\partial_{2}f\right)_{i,j}&:=\left\{\begin{array}{cl}
    	\frac{f_{i-1,j}-f_{i,j}}{h}, & i>1\\
    	0 & i=1\end{array}\right.\approx \frac{\partial}{\partial x_2} f(x_{i,j}).
    \end{align*}

\subsection{Construction of the TV, NID and A-NID regularizer}

The primal-dual algorithm is the standard way to implement the $L_2$-TV regularization. In order to be consistent with our gradient descent approach, we consider the regularized TV-functional proposed by Acar and Vogel in \cite{acar94},
	\begin{equation*}
		\T_{TV}(f)=\frac{1}{2}\norm[\A f-g]_{\Y}^2+\beta \int_{\Omega}\sqrt{\varepsilon+\norm[\nabla f]_2^2}\rmd x.
	\end{equation*}
In the simulations, we chose $\beta = 0.5h^2$ and $\epsilon=0.01$.\\

\par	
The general NID and A-NID approaches give flexibility on the choice and construction of the penalty terms. In the following, we consider three different approaches, but many more are possible:

	\begin{enumerate}
		\item[NID 1]: We choose $\frac{\gamma_k}{2}p_k(\cdot)$ to be 
		\begin{equation*}
			\frac12\gamma_k \lambda_k^2\log\left(1+\frac{\norm[\nabla_\sigma f]^2_2}{\lambda_k^2}\right),
		\end{equation*}
		which corresponds to a weighted average of the anti-derivative of the diffusion rate proposed by Perona and Malik for different thresholds.
		\item[NID 2]: We use shifted versions of Perona-Malik-type diffusion functions of the form 
		\begin{equation*}
			\varphi_k(s^2)=\left\{\begin{matrix}
				\dfrac{\lambda_k^2}{\lambda_k^2+(s-s_k)^2}, &s\geq s_k\\0,& s<s_k
			\end{matrix}\right. 
		\end{equation*}
		with $s_k$ a given shift. This can be interpreted as applying the diffusion process on different grey value scales and computing a weighted mean.
		\item[NID 3]: Here we build the derivative of the flux function directly to be a smooth alternating-pattern mimicking the forward-backward-diffusion switching of the previous approach, only with a more precise control of intervals where forward (or backward) diffusion occurs. The corresponding diffusion rate was then deduced from the derivative of the flux function evaluated by a quadrature scheme.
	\end{enumerate}
These three different approaches for the NID functional are depicted in  \Cref{fig:flux-functions-nid-large} and were computed based on the values of the Shepp-Logan phantom. We note them  $\psi^{(1)}$, $\psi^{(2)}$ and $\psi^{(3)}$ respectively.\\

\par As described in \Cref{sec:a-nid}, the A-NID approach consists in allowing the NID functional to change during the iteration process.  We gave a special structure to the A-NID by considering a TV functional ${\T_{TV}}$ as starting point and then converging to a target NID functional ${\T_{NID}}$. The combination of the TV and NID functionals is controlled by a sigmoid function $w(n)$, see \Cref{fig:omega-functions-anid-large}. Furthermore, we added a scalar function $\zeta(n)$, see \Cref{fig:omega-functions-anid-large}, in order to adapt the scale of the flux function leading to the proposed construction:
$$
(1-\omega(n)) {\T_{NID}} (\zeta(n) f_n)+\omega(n){\T_{TV}}(f_n).
$$
The impact of the scaling function $\zeta$ on ${\T_{NID}} (\zeta(n) \cdot)$  is depicted in \Cref{fig:flux-functions-anid-large} for $n=1$ and in \Cref{fig:flux-functions-anid-large-1000} for $n=1000$.

\begin{figure}[t]\centering
\includegraphics[width=\linewidth]{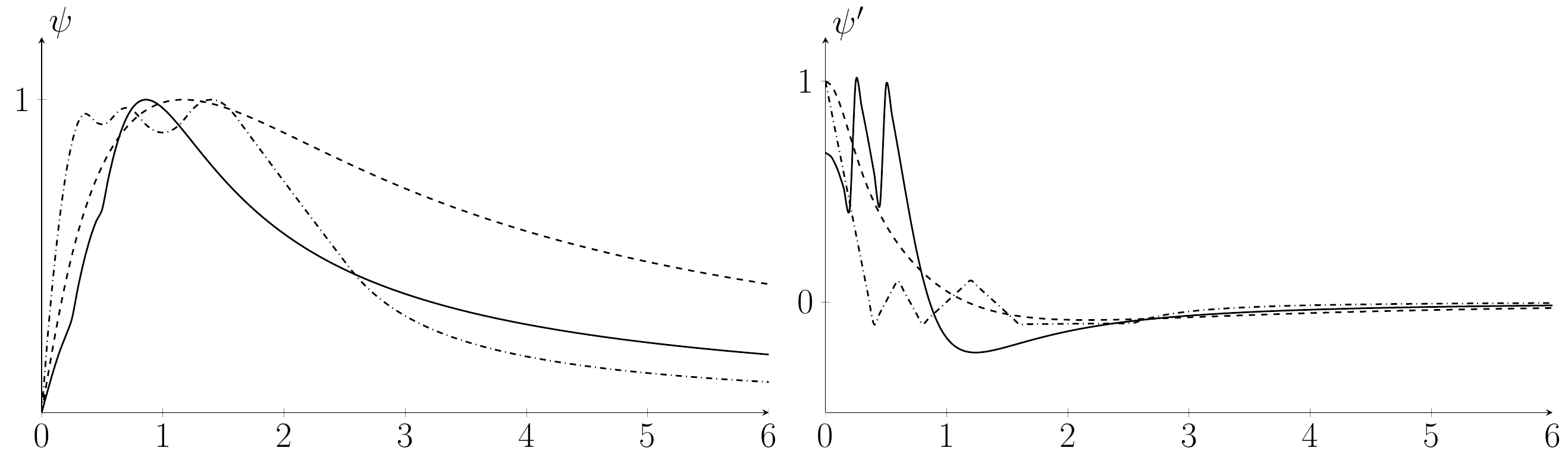}		
\caption{Normalized flux functions used for the NID regularization: $\psi^{(1)}$ (dashed), $\psi^{(2)}$ (solid) and $\psi^{(3)}$ (dashdotted) (left) as well as their rescaled derivatives (right). \label{fig:flux-functions-nid-large}}
\end{figure}
	
\begin{figure}[t]\centering
\includegraphics[width=0.7\linewidth]{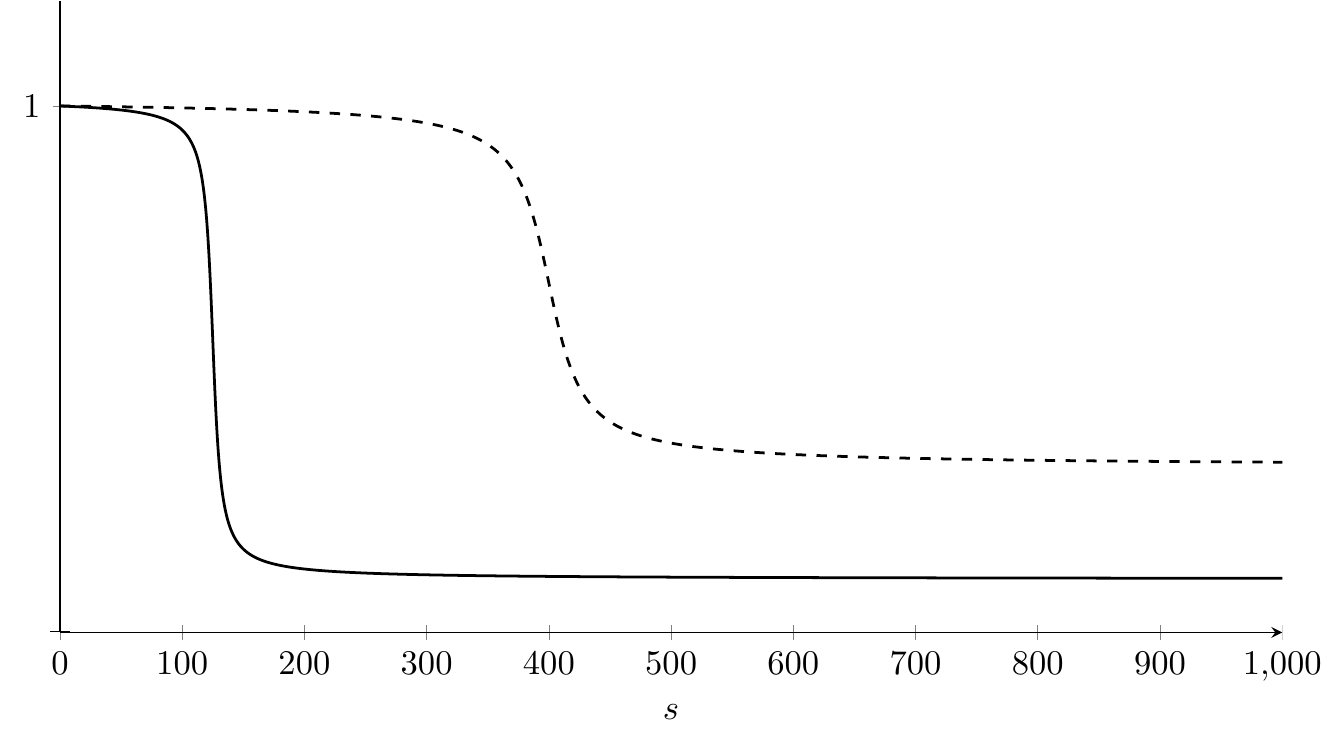}
\caption{Normalized switching function $\omega$ (solid) and scaling function $\zeta$ (dashed).}
\label{fig:omega-functions-anid-large}
\end{figure}

\begin{figure}[t]\centering
\includegraphics[width=\linewidth]{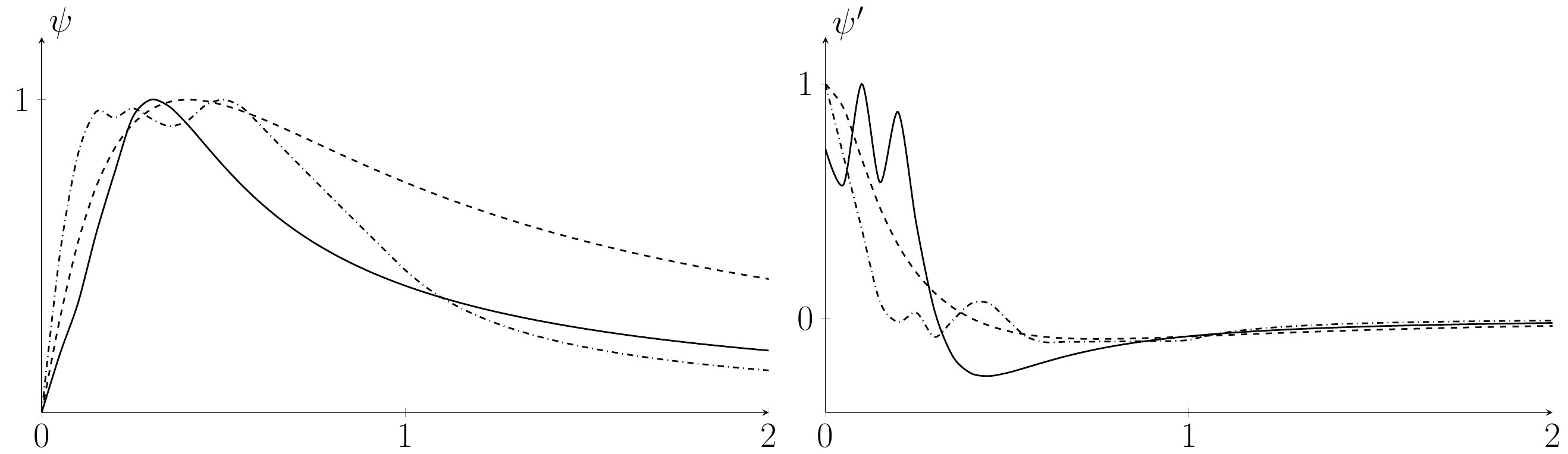}			
\caption{Normalized flux functions used for the A-NID results at the start 		$\psi^{(1)}$ (dashed), $\psi^{(2)}$ (solid) and $\psi^{(3)}$ (dashdotted) (left) as well as their rescaled derivatives (right).}
\label{fig:flux-functions-anid-large}
\end{figure}
	
\begin{figure}[t]\centering
\includegraphics[width=\linewidth]{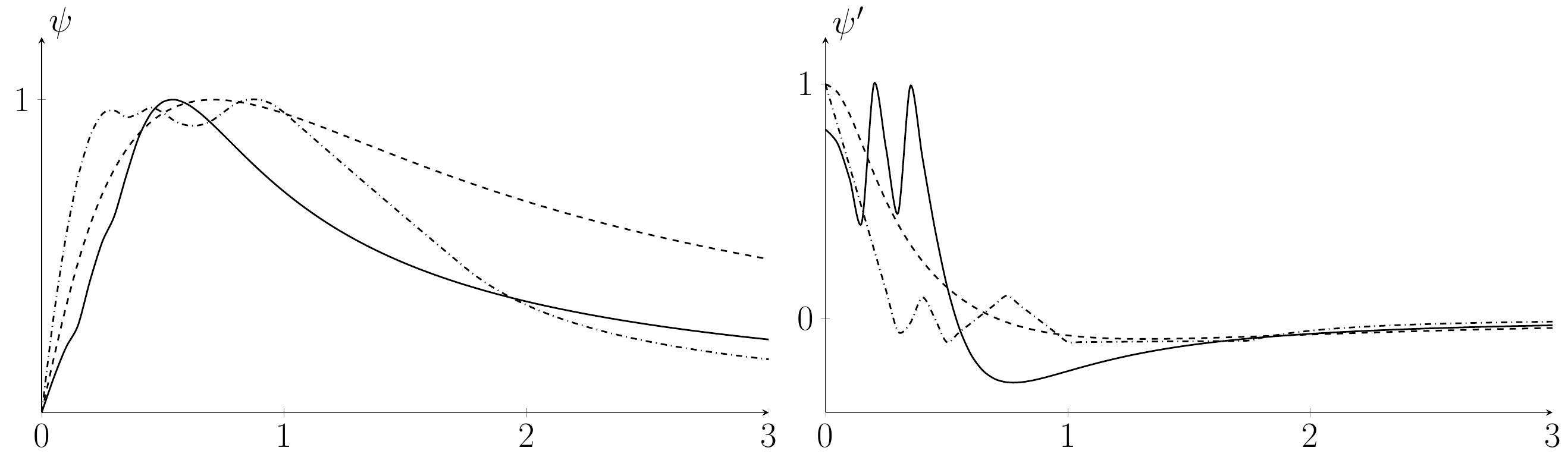}		
\caption{Normalized flux functions used for the A-NID results after 1000 iterations $\psi^{(1)}$ (dashed), $\psi^{(2)}$ (solid) and $\psi^{(3)}$ (dashdotted) (left) as well as their rescaled derivatives (right).}		\label{fig:flux-functions-anid-large-1000}
\end{figure}

\subsection{Simulation results and discussion}

The simulations were performed to solve the inverse problem $\A f =g$ associated to the Radon transform in the case of the Shepp-Logan phantom. We consider the noisy data $g_\delta$ given in \Cref{fig:sino}. Following on from the construction of the TV, NID and A-NID functionals as described above, we first considered the simple NID functional. These results are given in \Cref{fig:reconstructions-nid-large}. As expected the FBP reconstruction provides a noisy reconstruction in spite of the Shepp-Logan filter. The TV reconstruction shows a satisfactory reconstruction but the edges in the reconstruction are not sharp. The three NID reconstructions were implemented using \ref{alg:alg-final} and deliver a sharp reconstruction leading to very high SNR and PSNR, see \Cref{tab:comparison-results-large}. However, the small details like the shape of the tumors are not perfectly defined. \\

\par 
In order to find a compromise between sharpness and good localization of the different features, we implemented \ref{alg:alg-mixed} for the proposed A-NID functional. The results are given in \Cref{fig:reconstructions-anid-large}. We observe that the tumors (lower ellipses) are better reconstructed and located while preserving a high SNR and PSNR, see \Cref{tab:comparison-results-large}. The general visual impressions given by the reconstructions are validated by the structural similarity (SSIM) in \Cref{tab:comparison-results-large}.

	\begin{table}[h]
	\begin{tabular}{c|cccccccc}
	    & FBP   & TV    & NID 1 & NID 2 & NID 3 & A-NID 1 & A-NID 2 & A-NID 3\\\hline
	snr & 9.20  & 15.52 & 17.26 & 17.62 & 16.24 & \textbf{18.04} & 17.61 & 16.44\\
	psnr& 21.37 & 27.69 & 29.44 & 29.79 & 28.42 & \textbf{30.21} & 29.78 & 28.61\\
	ssim& 0.61  & 0.9   & 0.91  & 0.92  & 0.93  & 0.92  & 0.93  & \textbf{0.94}\\
	\end{tabular}
	\caption{Different error measures for the reconstructed images: TV and NID functionals were implemented by \Cref{alg:alg-final} while the A-NID functionals were implemented by \Cref{alg:alg-mixed}.}
	\label{tab:comparison-results-large}
	\end{table}

	\begin{figure}[t]
		\centering
		\begin{subfigure}[b]{6cm}
			\centering
			\includegraphics[width=2.8cm]{orig_image}
			\includegraphics[width=2.8cm]{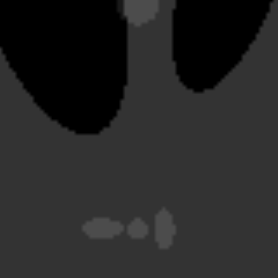}
			\caption{Original image}\label{fig:nid-a-large}
		\end{subfigure}
		\begin{subfigure}[b]{6cm}
			\centering
			\includegraphics[width=2.8cm]{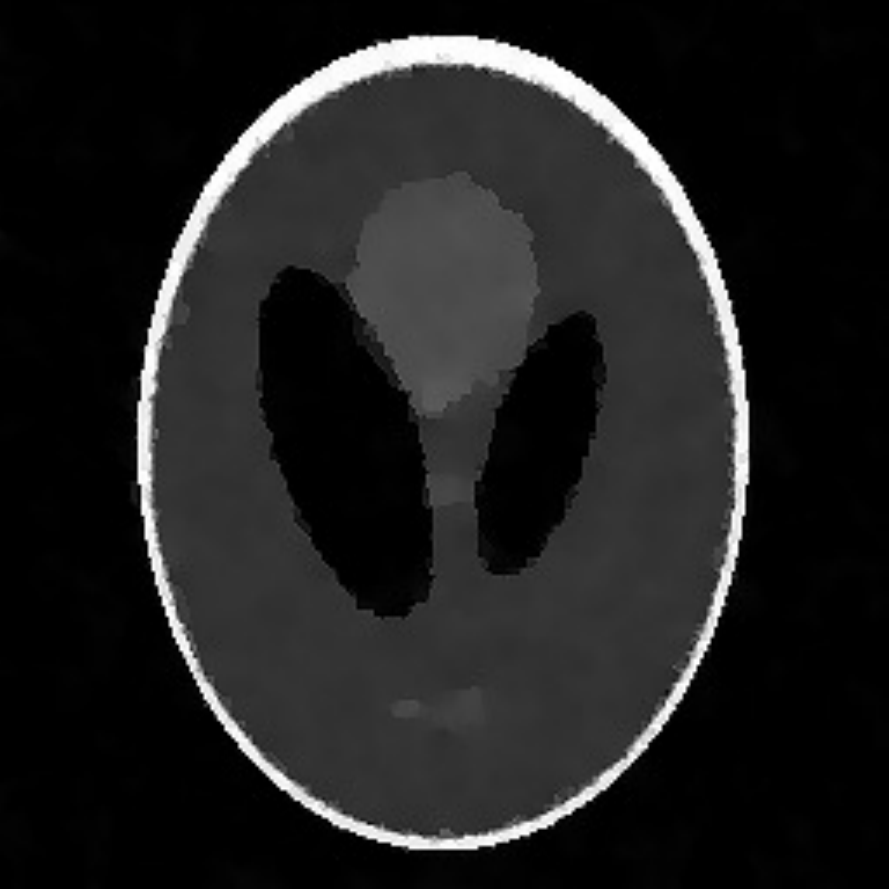}
			\includegraphics[width=2.8cm]{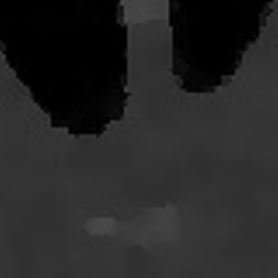}
			\caption{Reconstruction for $\phi^{(1)}$}\label{fig:nid-b-large}
		\end{subfigure}\\	
		\begin{subfigure}[b]{6cm}
			\centering
			\includegraphics[width=2.8cm]{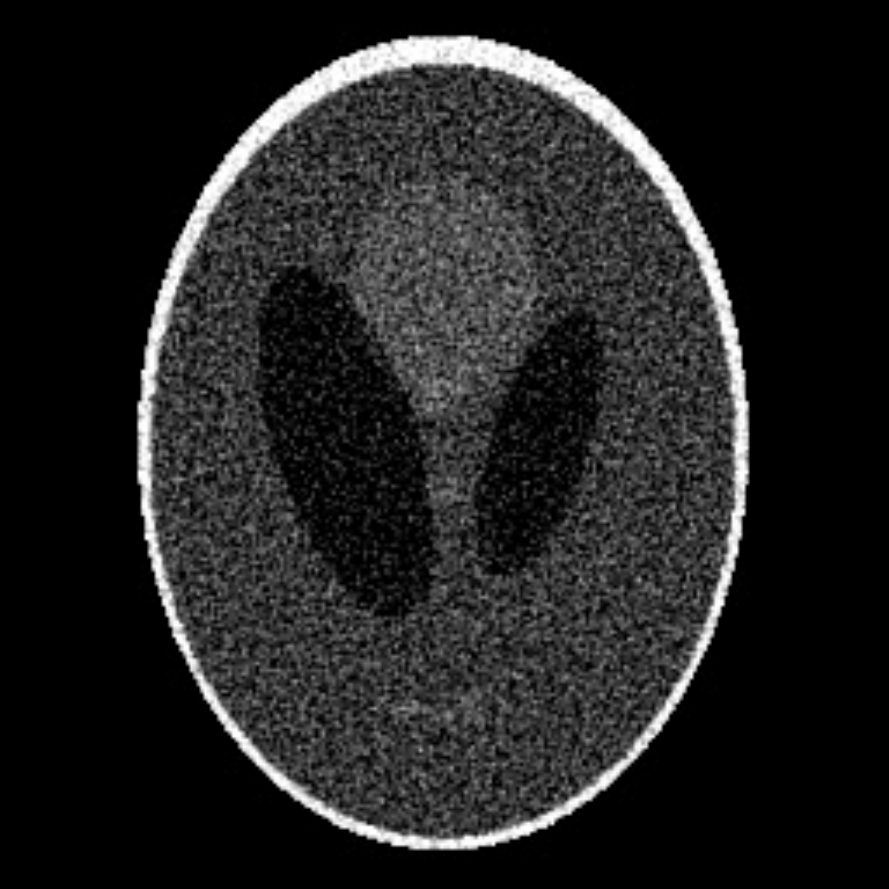}
			\includegraphics[width=2.8cm]{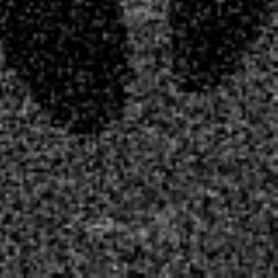}
			\caption{FBP reconstruction}\label{fig:nid-c-large}
		\end{subfigure}
		\begin{subfigure}[b]{6cm}
			\centering
			\includegraphics[width=2.8cm]{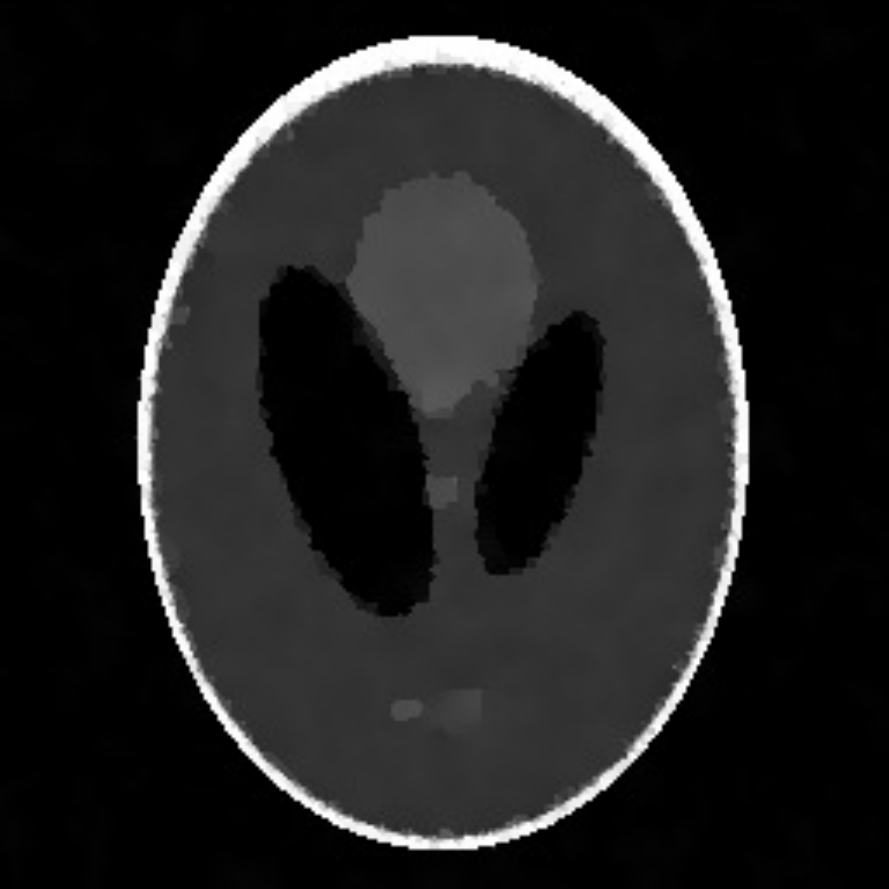}
			\includegraphics[width=2.8cm]{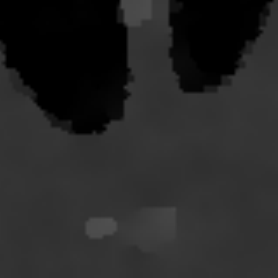}
			\caption{reconstruction for $\phi^{(2)}$}\label{fig:nid-d-large}
		\end{subfigure}\\	
		\begin{subfigure}[b]{6cm}
			\centering
			\includegraphics[width=2.8cm]{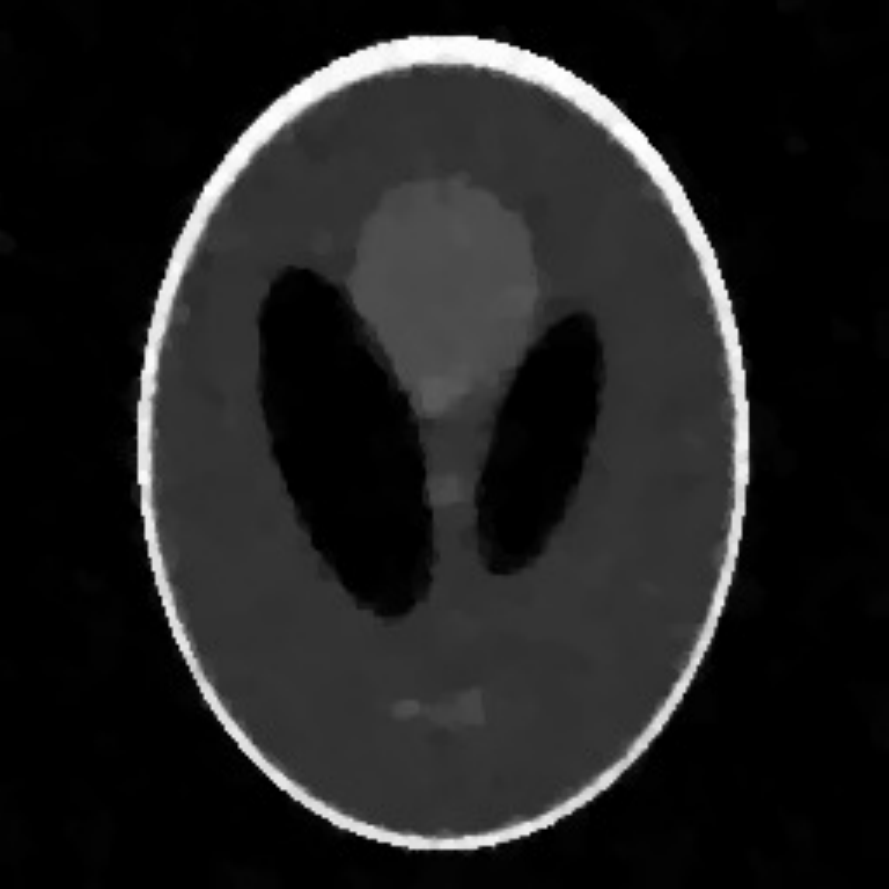}
			\includegraphics[width=2.8cm]{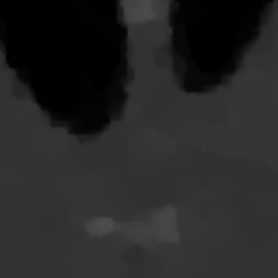}
			\caption{TV reconstruction}\label{fig:nid-e-large}
		\end{subfigure}
		\begin{subfigure}[b]{6cm}
			\centering
			\includegraphics[width=2.8cm]{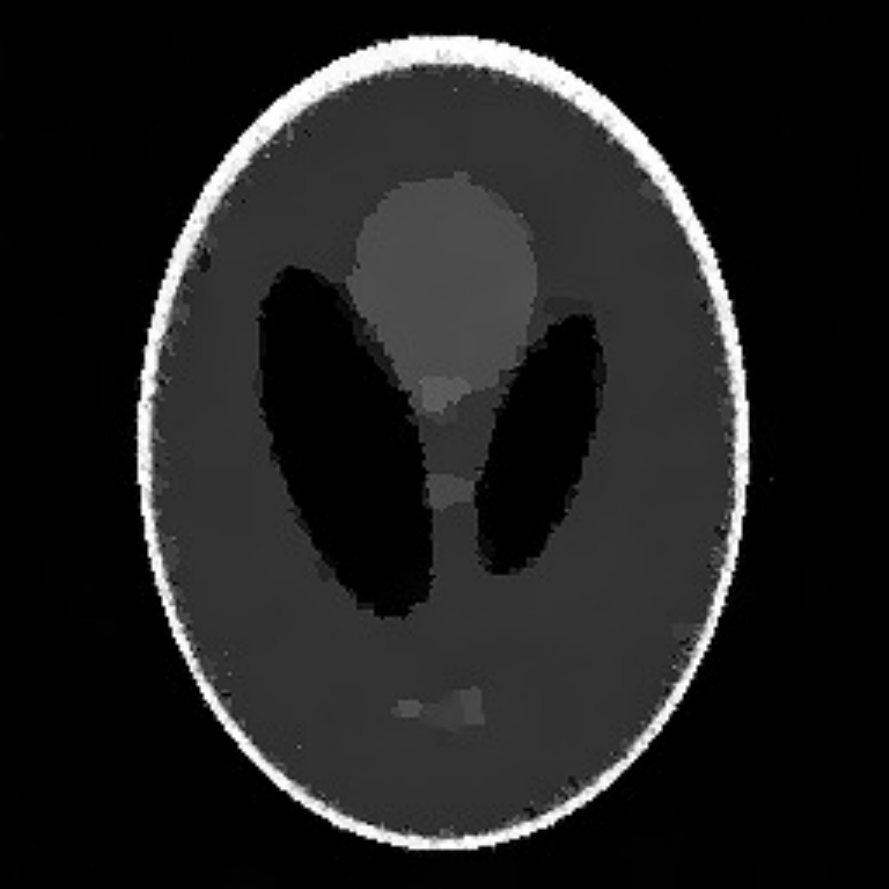}
			\includegraphics[width=2.8cm]{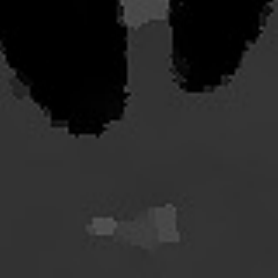}
			\caption{Reconstruction for $\phi^{(3)}$}\label{fig:nid-f-large}
		\end{subfigure}
		\vspace*{-2ex}
		\caption{Comparison of different NID reconstruction methods.}
		\label{fig:reconstructions-nid-large}
	\end{figure}
	
	\begin{figure}[t]
		\centering
		\begin{subfigure}[b]{6cm}
			\centering
			\includegraphics[width=2.8cm]{orig_image}
			\includegraphics[width=2.8cm]{orig_image_zoom}
			\caption{Original image}\label{fig:anid-a-large}
		\end{subfigure}
		\begin{subfigure}[b]{6cm}
			\centering
			\includegraphics[width=2.8cm]{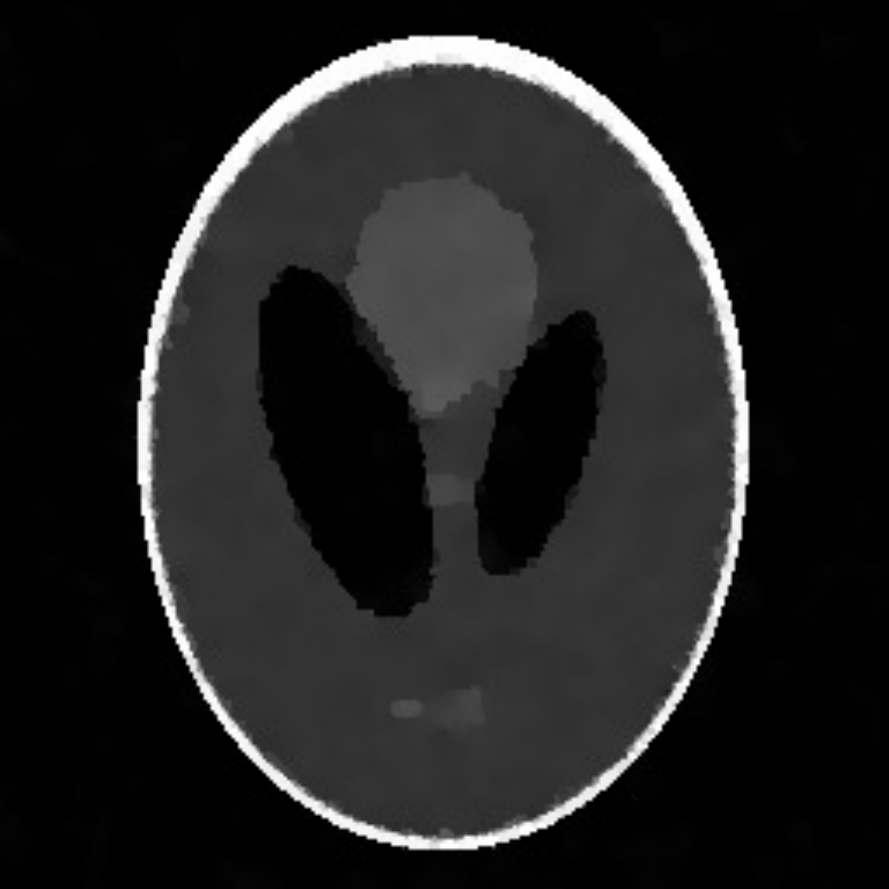}
			\includegraphics[width=2.8cm]{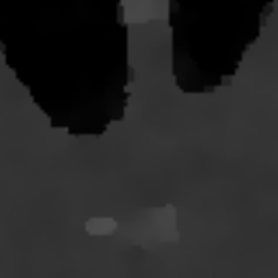}
			\caption{Reconstruction for $\phi^{(1)}$}\label{fig:anid-b-large}
		\end{subfigure}\\	
		\begin{subfigure}[b]{6cm}
			\centering
			\includegraphics[width=2.8cm]{fbp}
			\includegraphics[width=2.8cm]{fbp_zoom}
			\caption{FBP reconstruction}\label{fig:anid-c-large}
		\end{subfigure}
		\begin{subfigure}[b]{6cm}
			\centering
			\includegraphics[width=2.8cm]{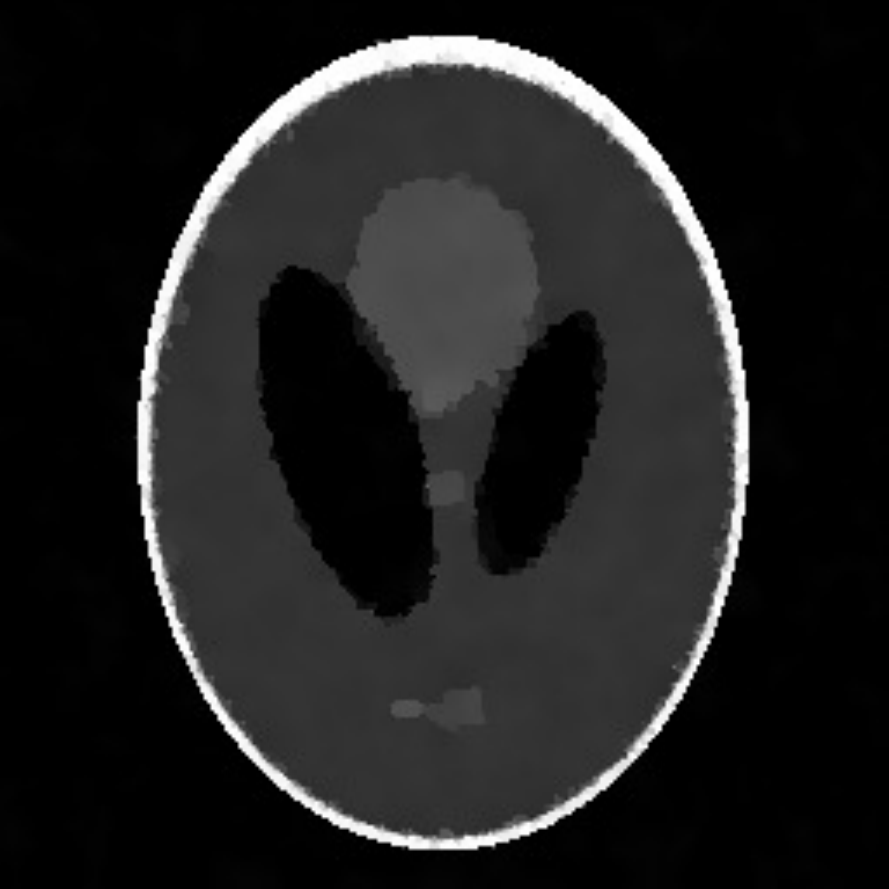}
			\includegraphics[width=2.8cm]{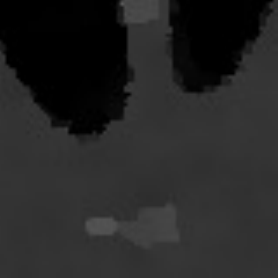}
			\caption{reconstruction for $\phi^{(2)}$}\label{fig:anid-d-large}
		\end{subfigure}\\	
		\begin{subfigure}[b]{6cm}
			\centering
			\includegraphics[width=2.8cm]{l2tv}
			\includegraphics[width=2.8cm]{l2tv_zoom}
			\caption{TV reconstruction}\label{fig:anid-e-large}
		\end{subfigure}
		\begin{subfigure}[b]{6cm}
			\centering
			\includegraphics[width=2.8cm]{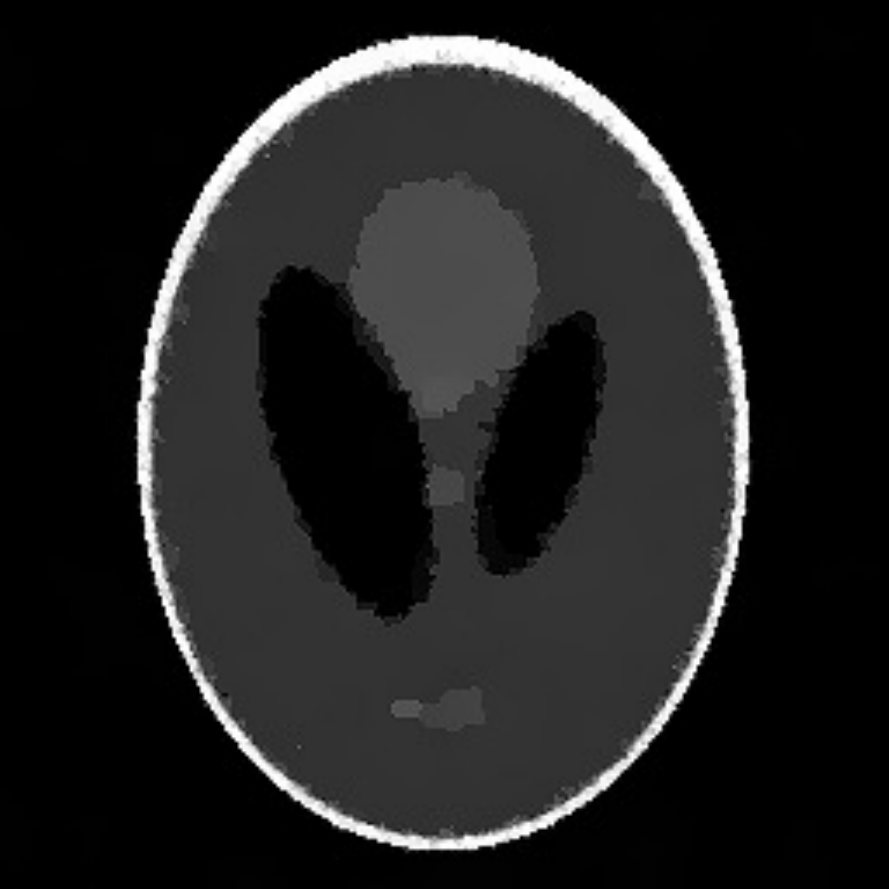}
			\includegraphics[width=2.8cm]{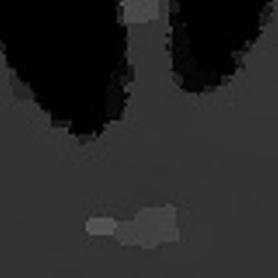}
			\caption{Reconstruction for $\phi^{(3)}$}\label{fig:anid-f-large}
		\end{subfigure}
		\vspace*{-2ex}
		\caption{Comparison of different reconstruction methods for the ANID case.}
		\label{fig:reconstructions-anid-large}
	\end{figure}


	\bibliographystyle{siamplain}
	\bibliography{Literature.bib}

\appendix
\section{ Proof of \Cref{thm:continuity-derivative} }\label{proof:continuity-derivative}

Before to go further, we define the $L_2$-norm for vector-valued functions as well as the $C^1$-norm, 
		 $$
		 \norm[v]_{L_2(\Omega,\mathbb{R}^2)}:=\left(\int_\Omega \norm[v(x)]^2_2\rmd x\right)^{1/2}, \qquad 
		 \norm[v]_{C^1(\overline{\Omega})}:=\sup_{x\in\overline{\Omega}}\norm[v(x)]_2.
		 $$
		 
\begin{enumerate}[label=\arabic*)]
    \item We want first to prove that the functional is Fr\'echet differentiable. For $\frac{1}{2}\norm[ \A f-g]^2_{\Y}$ and $\frac{\alpha}{2}\norm[f]_{ L_2(\Omega)}^2$ it is clear. Focusing on the NID term, we have that the operator $\nabla_{\sigma_k}:L_2(\Omega)\to L_2(\Omega,\mathbb{R}^2)$
is Fr\'echet differentiable since it is linear and bounded. Let $\tilde{f}\in L_2$ and consider for all $k$
		\begin{equation*}
			K_{\tilde{f}}^k:=K^k(\tilde{f}):L_2(\Omega,\mathbb{R}^2)\to \mathbb{R},\qquad \tilde{f}\mapsto\frac{1}{2}\int_{\Omega}p_k\left(\norm[\tilde{f}]_2^2\right)\rmd x.
		\end{equation*}
		As usual $p_k'=\varphi_k$ denotes the associated diffusion function. Also for all $k$ the operator
		\begin{equation*}
			\tilde{K}^k_{\tilde{f}}:L_2(\Omega,\mathbb{R}^2)\to\mathbb{R},\qquad u\mapsto\int_{\Omega}\varphi_k\left(\norm[\tilde{f}]_2^2\right)\scalarproduct{\tilde{f}}{u}_2\rmd x
		\end{equation*}
		is linear and bounded for fixed $\tilde{f}$ since $\varphi_k$ is bounded. Under our conditions on $p_k$ (and implicitly $\varphi_k$), $\phi:\mathbb{R}^2\to\mathbb{R}$, $\phi(\xi):= \frac{1}{2}p_k(\norm[\xi]^2_2)$ is twice continuously differentiable, satisfies $\frac{\partial}{\partial \xi}\phi(v)= \varphi_{k}(\norm[v]^2_2)v^T$ and furthermore, the Hessian stays bounded on $\mathbb{R}$ as a consequence of the chain rule.	We consequently obtain with the Taylor's theorem for multivariate functions (see Theorem 5.8 and Remark 5.9, section VII in \cite{amann06}):
			\begin{align}
					&\left|\phi(v+w)-\phi(v)-\frac{\partial}{\partial \xi}\phi\left(v\right)w\right|\nonumber\\
					&=\scalarproduct{w}{\frac{\partial^2}{\partial \xi^2}\phi\left(v\right)w}_2+ \int_{0}^{1}\scalarproduct{w}{\left(\frac{\partial^2}{\partial \xi^2}\phi\left(v+t w\right)-\frac{\partial^2}{\partial \xi^2}\phi\left(v\right)\right)w}_2\rmd t\nonumber\\
					&\leq C\norm[w]_{2}^2\label{eq:taylor-ineq}
			\end{align}
			for some $C>0$ independent of $v,w$, which shows that 
			\begin{equation*}
				\left|K_{\tilde{f}+h}^k-K_{\tilde{f}}^k-\tilde{K}_{\tilde{f}}^k h\right|\leq C\norm[h]_{L_2(\Omega,\mathbb{R}^2)}^2
			\end{equation*}	
			for any search direction $h\in L_2(\Omega,\mathbb{R}^2)$, \textit{i.e.} $\tilde{K}_f^k$ is the Fr\'{e}chet-derivative of $K^k$. Using the chain rule for the Fr\'{e}chet-derivative (see \textit{e.g.} \cite[Theorem III.5.4]{werner11}), we find that $K^k\circ \nabla_{\sigma_k}$ is Fr\'{e}chet-differentiable for every $f\in L_2(\Omega)$ with for all $\tilde{h}\in L_2(\Omega)$
			\begin{equation*}
			    \left(K^k\circ \nabla_{\sigma_k}\right)'(f)\tilde{h}=\tilde{K}_{\nabla_{\sigma_k} f}^k K_{\nabla_{\sigma_k} f}^k \tilde{h} =\scalarproduct{\varphi_{k}\left(\norm[\nabla_{\sigma_k} f]_2^2\right)\nabla_{\sigma_k}f}{\nabla_{\sigma_k} \tilde{h}}_{L_2(\Omega)}.
			\end{equation*}
			By linearity of the Fr\'echet derivative \cref{eq:frech-deriv-func} follows.
			\item In addition, for any $u\in \mathbb{R}^2 $ we have
			 \begin{align*}			    
			 	&\left| \frac{\partial}{\partial \xi}\phi(v+w)u -\frac{\partial}{\partial \xi}\phi(v)u\right|\\
			 		\leq &\norm[\frac{\partial}{\partial \xi}\phi(v+w)-\frac{\partial}{\partial \xi}\phi(v)]_2\norm[u]_2\\	
			 		= &\norm[\frac{\partial^2 }{\partial \xi^2}\phi(v)w+\int_{0}^1\left(\frac{\partial^2 }{\partial \xi^2}p(v+t w)-\frac{\partial^2 }{\partial \xi^2}\phi(v)\right)w\,\rmd t]_2\norm[u]_2\\
		 			\leq &\, C \norm[w]_2 \norm[u]_2.
		 	\end{align*} 
		 	Consequently by applying the  H\"older's inequality, it holds for any $\tilde{u}\in L_2(\Omega,\mathbb{R}^2)$	
		 	\begin{align*}
		 		&
		 		\left| \int_{\Omega} \scalarproduct{\varphi_k\left(\norm[\tilde{f}+h]^2_2\right)(\tilde{f}+h)-\varphi_k\left(\norm[\tilde{f}]^2_2\right)\left(\tilde{f}\right)}{\tilde{u}}_2\rmd x \right|\\
		 		&\quad\leq C\int_{\Omega}\norm[h]_2\norm[u]_2\rmd x \leq C\norm[h]_{L_2(\Omega,\mathbb{R}^2)}\norm[\tilde{u}]_{L_2(\Omega,\mathbb{R}^2)}
		 	\end{align*}
            This implies (since $\Omega$ is bounded) that the operator norm of $\tilde{K}_{\tilde{f}+h}^k-\tilde{K}^k_{\tilde{f}}$ satisfies 
		 \begin{equation}
		 	\label{eq:cont-eq2}
		 	\norm[\tilde{K}_{\tilde{f}+h}^k-\tilde{K}_{\tilde{f}}^k]_{\mathcal{L}(L_2(\Omega),\mathbb{R})}\leq C \norm[h]_{L_2(\Omega,\mathbb{R}^2)}
		 \end{equation}
		 implying that
		 \begin{align*}
		 	\norm[\tilde{K}_{\nabla_{\sigma_k} f_1}^k-\tilde{K}_{\nabla_{\sigma_k} f_2}^k]_{\mathcal{L}(L_2(\Omega),\mathbb{R})}
		 	&\leq C\norm[\nabla_{\sigma_k}f_1-\nabla_{\sigma_k}f_2]_{L_2(\Omega,\mathbb{R}^2)}\\
		 	&\leq \tilde{C} \norm[f_1-f_2]_{L_2(\Omega)}
		 \end{align*}
		 for some $\tilde{C}>0$. Since
		 \begin{align*}
				&\T_{g,\gamma,\sigma,\alpha}'(f_1)-\T_{g,\gamma,\sigma,\alpha}'(f_2)=
				\A^*( \A f_1-f_2)+\alpha (f_1-f_2)\\
				&\quad-\sum_{k=1}^{K}\gamma_k(\nabla_{\sigma_k})^*\left(\varphi_k\left(\norm[\nabla_{\sigma_k} f_1]_2^2\right)\nabla_{\sigma_k} f_1-\varphi_k\left(\norm[\nabla_{\sigma_k} f_2]_2^2\right)\nabla_{\sigma_k} f_2\right)
			\end{align*}
			and $(\nabla_{\sigma_k})^*$ is bounded, \cref{eq:frechet-lipschitz} follows.
		 \item  
		 Let $f_n\rightharpoonup f$. Since $\norm[\nabla_{\sigma_k}(f_n-f)(x)]_2\leq \norm[f_n-f]_{C^1(\overline{\Omega})}\forall x\in\Omega$, there exists a constant $\hat{C}>0$ such that
		 $$
		 \norm[\nabla_{\sigma_k} (f_n-f)]_{L_2(\Omega,\mathbb{R}^2)}\leq \hat{C} \norm[\nabla_{\sigma_k}(f_n-f)]_{C^1(\overline{\Omega})}.
		 $$
		 Furthermore, it holds $\norm[\nabla_{\sigma_k} (f_n-f)]_{L_2(\Omega,\mathbb{R}^2)} \to 0$ due the compactness of $\nabla_{\sigma_k}$. We now have for every  $\tilde{u}\in L_2(\Omega)$
		 \begin{align*}
		 	\left|\tilde{K}_{\nabla_{\sigma_k} f_n}^k \nabla_{\sigma_k} u -\tilde{K}_{\nabla_{\sigma_k} f}^k\nabla_{\sigma_k} \tilde{u}\right|&\leq \norm[\tilde{K}_{\nabla_{\sigma_k} f_n}^k-\tilde{K}_{\nabla_{\sigma_k} f}^k]\norm[\nabla_{\sigma_k} \tilde{u}]_{L_2(\Omega,\mathbb{R}^2)}\\&\leq \norm[\nabla_{\sigma_k}] \norm[\tilde{K}_{\nabla_{\sigma_k} f_n}^k-\tilde{K}_{\nabla_{\sigma_k}f}^k]\norm[\tilde{u}]_{L_2(\Omega)}.
		 \end{align*}
		 Using \cref{eq:cont-eq2} for $h:=\nabla_{\sigma_k}f_n-\nabla_{\sigma_k}f$ shows $\tilde{K}_{\nabla_{\sigma_k} f_n}^k\to\tilde{K}_{\nabla_{\sigma_k} f}^k$. Since $\A f_n\rightharpoonup \A f$, this proves \cref{eq:weak-weak-continuity}.	
	 \end{enumerate}
	\section{Proof of \Cref{thm:conv-anid} }\label{proof:conv-anid}
	\begin{enumerate}[label=\arabic*)]
		\item Similarly to \cref{eq:taylor-ineq}, we obtain with Taylor's theorem for $\phi(\xi)=\frac{1}{2}p_{k,n}\left(\norm[\xi]_2^2\right)$ ($k=1,...,K$, $n\in\mathbb{N}$),
		\begin{align*}
			\phi(v+w)&=\phi(v)+\frac{\partial}{\partial 	\xi}\phi\left(v\right)w+\scalarproduct{w}{\frac{\partial^2}{\partial 	\xi^2}\phi\left(v\right)w}_2+\\
			&\qquad\qquad \int_{0}^{1}\scalarproduct{w}{\left(\frac{\partial^2}{\partial \xi^2}\phi\left(v+t w\right)-\frac{\partial^2}{\partial \xi^2}\phi\left(v\right)\right)w}_2\rmd t
		\end{align*}
		for arbitrary $v,w\in \mathbb{R}^2$.
		Since $\frac{\partial^2}{\partial \xi^2}\phi(v)$ stays bounded, the inequality
		\begin{align*}
			\phi(v+w)&\leq \phi(v)+\frac{\partial}{\partial 	\xi}\phi\left(v\right)w+\kappa_k\norm[w]_2^2
		\end{align*}
		holds for $\kappa_k \geq 3L_k$. Consequently, for $k=1,...,K$
		\begin{align*}\allowdisplaybreaks
			&\frac{\gamma_k}{2}\int_{\Omega}p_{k,n}\left(\norm[\nabla_{\sigma_k}\left(f_n-\tau^l v_n\right)]_2^2\right)\rmd x \\
&\leq \frac{\gamma_k}{2}\int_{\Omega}p_{k,n}\left(\norm[\nabla_{\sigma_k}\left(f_n\right)]_2^2\right)\rmd x +\gamma_k\int_{\Omega}\kappa_k \norm[\nabla_{\sigma_k}]^2\norm[v_n]_2^2\rmd x\\
			&\quad -\gamma_k\int_{\Omega}\varphi_{k,n}\left(\norm[\nabla_{\sigma_k}f_n]_2^2\right)\left(\nabla_{\sigma_k}f_n\right)^T\nabla_{\sigma_k} v_n\rmd x .
		\end{align*} 
		An analogous inequality holds for $p_{AV}$, $\varphi_{AV}$ with some $\kappa_{TV}>0$.
		Furthermore, it holds
		\begin{align*}\allowdisplaybreaks
			&\frac{1}{2}\norm[\A (f_n-\tau^lv_n)-g]^2_{L_2(\tilde{\Omega})}+\frac{\alpha}{2}\norm[f_n-\tau^lv_n]^2_{L_2(\Omega)}\\
			\leq& \frac{1}{2}\norm[\A f_n-g]^2_{L_2(\tilde{\Omega})}+\frac{\alpha}{2}\norm[f_n]^2_{L_2(\Omega)}\\
			& -\tau^l \scalarproduct{\A^*(\A v_n-g)}{v_n}_{L_2(\Omega)}-\alpha \tau^l\scalarproduct{f_n}{v_n}_{L_2(\Omega)}\\
			&+\frac{1}{2}\tau^{2l}\norm[\A v_n]_{L_2(\tilde{\Omega})}+\frac{\alpha}{2} \tau^{2l}\norm[v_n]_{L_2(\Omega)}^2.
		\end{align*} 
		Thereby, the following inequalities hold
		\begin{align*}
		    \T_{NID}^n(f_n-\tau^l v_n) \leq&\T_{NID}^n (f_n)-\left(\T_{NID}^{n}\right)' (f_n)v_n\\
		    &+
		    \frac{1}{2}\tau^{2l}\norm[\A v_n]_{L_2(\tilde{\Omega})}+\frac{1}{2}\tau^{2l}\sum_{k=1}^K \kappa_k\gamma_k\norm[v_n]^2+\frac{\alpha}{2} \tau^{2l}\norm[v_n]_{L_2(\Omega)}^2 ,\\
		    \T_{TV}(f_n-\tau^l v_n) \leq&\T_{TV} (f_n)-\T_{TV}'(f_n)v_n\\
		    &+\frac{1}{2}\tau^{2l}\norm[\A v_n]_{L_2(\tilde{\Omega})}+\frac{1}{2}\tau^{2l}\kappa_{TV}\beta \norm[v_n]^2 +\frac{\alpha}{2} \tau^{2l}\norm[v_n]_{L_2(\Omega)}^2 .
		\end{align*}
		Since $\kappa_k$ is independent from $n$ for $k=1,...,N$, we now find $\tilde{\kappa}>0$ also independent from $n$ such that
		\begin{align*}
		\T_n(f_{n+1})&=T_n(f_{n}-\tau^lv_n)\\
		&=(1-\omega(n))\T_{NID}^n(f_{n}-\tau^lv_n)+\omega(n)\T_{TV}(f_{n}-\tau^lv_n)\\
		&\leq \T_n(f_{n})-t_n\T_n'(f_n)v_n+\frac{1}{2}\kappa \tau^{2l}\norm[v_n]^2\\
		&=\T_n(f_{n})+\tau^l\left(\frac{1}{2}\tilde{\kappa }\tau^l-1\right)\norm[v_n]^2.
		\end{align*}
		Since $\norm[v_n]>0$ and $\frac{1}{2}\kappa \tau^l-1$ converges to $-1<-\mu$ as $l\to\infty$, \cref{eq:mixed_step_cond} holds as soon as
		$$
		\frac{1}{2}\tilde{\kappa} \tau^l-1\leq-\mu \quad\Leftrightarrow \quad			\tau^l\leq \frac{2(1-\mu)}{\tilde{\kappa}}
		$$
		Setting $\theta:=\min\left(\tau,\frac{2(1-\mu)}{\tilde{\kappa}\tau}\right)$ proves the assertion.
		\item Let $n>N$. It holds that
		\begin{align*}
		\T_{n+1}(f_{n+1})-
		\T_{n}(f_{n})&=\underbrace{\left(
			\T_{n+1}(f_{n+1})-
			\T_{n}(f_{n+1})\right)}_{=:I_1}+\underbrace{\left(
			\T_{n}(f_{n+1})-
			\T_{n}(f_{n})\right)}_{=:I_2}.
		\end{align*}
		Due to \cref{eq:wolfe-1}, $I_2$ is smaller than zero, in particular,
		\begin{equation}
		\label{eq:smaller_than}
		I_2\leq -\mu t_n\norm[v_n]^2.
		\end{equation}	
		For $I_1$ it holds by construction of the functional, see \cref{eq:condition_nid_nk_2},
		\begin{align*}
		I_1&=\left(1-\omega(n+1)\right)\T_{NID}^{n+1}(f_{n+1})+\omega(n+1)\T_{TV}(f_{n+1})\\
		&\qquad-\left(1-\omega(n)\right)\T_{NID}^{n}(f_{n+1})-\omega(n)\T_{TV}(f_{n+1})\\
		&=\left(\omega(n)-\omega(n+1)\right) \left(\T_{NID}^{n}(f_{n+1}) - \T_{TV}(f_{n+1}) \right) \\ 
	&+ (1-\omega(n+1)) \left(\T_{NID}^{n+1}(f_{n+1})  - \T_{NID}^{n}(f_{n+1}) \right) \ < \ 
		\mu t_n\norm[v_n]^2.
		\end{align*}
		Together, this shows $I_1+I_2<0$.
		\item As shown in 2), $(\T_n(f_n))_{n>N}$ is decreasing, which implies boundedness. Since 
		$\T_n(f_n)\geq \varepsilon \T_{NID}^n(f_n)$ for sufficiently large $n$, the uniform coercivity of $\T_{NID}^n$ shows that $(f_n)_{n}$ is bounded. The existence of an accumulation point is given by $\Sob^1(\Omega)$ being reflexive which ends the proof.
		\item $(\T_n(f_n))_{n>N}$ is monotonically decreasing and bounded from below. Consequently, the limit
		\begin{align*}
		T:=\lim_{n\to\infty} \T_n(f_n)
		\end{align*}
		exists.
		On the one hand, we have $\T_n(f_{n+1})\leq T_n(f_n)$, implying
		$$\limsup\limits_{n\to\infty}~\T_n(f_{n+1})\leq~T.$$ On the other hand, since $I_2 \leq -\mu t_n\norm[v_n]^2$ and $I_1<\mu t_n\norm[v_n]^2$, we have for $n>N$:
		\begin{align*}
		\T_{n+1}(f_{n+1})< \T_{n}(f_{n+1})+\mu t_n\norm[v_n]^2\leq\T_n(f_n).
		\end{align*}
		Using a Sandwich argument we can deduce that
		\begin{align*}
		    \lim\limits_{n\to \infty}\T_{n}(f_{n+1})=T.
		\end{align*}
		We now have
		\begin{align*}
		0\leq \mu t_n\norm[v_n]^2\leq \T_{n}(f_{n})-\T_{n}(f_{n+1})\to 0
		\end{align*}
		which completes the proof since $t_n$ is bounded from below according to 1).
	\end{enumerate}

\end{document}